\documentclass[11pt,reqno,tbtags]{amsart}

\usepackage{diagrams}   
\usepackage{graphicx}

\usepackage{amssymb}
\numberwithin{equation}{section}

\newtheorem*{property*}{Property \csname @currentlabel\endcsname}

\newtheorem{theorem}{Theorem}[section]
\newtheorem{lemma}[theorem]{Lemma}

\newtheorem{corollary}[theorem]{Corollary}

\theoremstyle{definition}
\newtheorem{Alg}[theorem]{Algorithm}
\newtheorem{example}[theorem]{Example}

\newtheorem{remark}{Remark}[section]
\newtheorem*{remark*}{Remark}
\newtheorem*{remarks}{Remarks}

\newtheorem*{acks}{Acknowledgements}

\theoremstyle{remark}

\newenvironment{romenumerate}{\begin{enumerate}
 }{\end{enumerate}}

\newenvironment{eqenumerate}{\begin{enumerate}\setcounter{enumi}{\value{equation}}
 }{\setcounter{equation}{\value{enumi}}\end{enumerate}}

\newcounter{oldenumi}
{\setcounter{oldenumi}{\value{enumi}}
\begin{romenumerate} \setcounter{enumi}{\value{oldenumi}}}
{\end{romenumerate}}

\newcounter{thmenumerate}
\newenvironment{thmenumerate}
{\setcounter{thmenumerate}{0}%
 \def\item{\par
 \refstepcounter{thmenumerate}\textup{(\roman{thmenumerate})\enspace}}
}
{}

\newcounter{xenumerate}   

\newcommand\pfitem[1]{\par(#1):}

\newcommand{\refT}[1]{Theorem~\ref{#1}}
\newcommand{\refC}[1]{Corollary~\ref{#1}}
\newcommand{\refL}[1]{Lemma~\ref{#1}}

\newcommand{\refS}[1]{Section~\ref{#1}}
\newcommand{\refSS}[1]{Subsection~\ref{#1}}

\newcommand{\refE}[1]{Example~\ref{#1}}
\newcommand{\refF}[1]{Figure~\ref{#1}}

\newcommand{\refAlg}[1]{Algorithm~\ref{#1}}
\newcommand{\refand}[2]{\ref{#1} and~\ref{#2}}

\begingroup
  \divide\count255 by 60
  \multiply\count255 by -60
  \advance\count255 by \time
  \ifnum \count255 < 10 \xdef\klockan{\the\count1.0\the\count255}
  \else\xdef\klockan{\the\count1.\the\count255}\fi
\endgroup

\newcommand\nopf{\qed}   

\newcommand\set[1]{\ensuremath{\{#1\}}}
\newcommand\bigset[1]{\ensuremath{\bigl\{#1\bigr\}}}

\newcommand\xpar[1]{(#1)}
\newcommand\bigpar[1]{\bigl(#1\bigr)}
\newcommand\Bigpar[1]{\Bigl(#1\Bigr)}

\newcommand\lrpar[1]{\left(#1\right)}

\def\rompar(#1){\textup(#1\textup)}    

\newcommand\Bigparfrac[2]{\Bigpar{\frac{#1}{#2}}}

\def\xexp(#1){e^{#1}}

\newcommand\floor[1]{\lfloor#1\rfloor}

\newcommand\ntoo{\ensuremath{{n\to\infty}}}
\newcommand\nntoo{\ensuremath{{n_1,n_2\to\infty}}}

\newcommand\iid{i.i.d.\spacefactor=1000}    
\newcommand\ie{i.e.\spacefactor=1000}
\newcommand\eg{e.g.\spacefactor=1000}

\newcommand\cf{cf.\spacefactor=1000}
\newcommand{\as}{a.s.\spacefactor=1000}
\newcommand{\aex}{a.e.\spacefactor=1000}

\newcommand\ii{\mathrm{i}}

\newcommand{\tend}{\longrightarrow}
\newcommand\dto{\overset{\mathrm{d}}{\tend}}
\newcommand\pto{\overset{\mathrm{p}}{\tend}}
\newcommand\asto{\overset{\mathrm{a.s.}}{\tend}}
\newcommand\eqd{\overset{\mathrm{d}}{=}}

\newcommand\bbR{\mathbb R}

\newcommand\bbN{\mathbb N}

\newcommand\bbZ{\mathbb Z}

\newcounter{CC} 
\newcounter{cc}

\newcommand\E{\operatorname{\mathbb E{}}}
\renewcommand\P{\operatorname{\mathbb P{}}}
\newcommand\Var{\operatorname{Var}}

\newcommand\Po{\operatorname{Po}}
\newcommand\Bin{\operatorname{Bin}}
\newcommand\Be{\operatorname{Be}}
\newcommand\Ge{\operatorname{Ge}}

\newcommand\ga{\alpha}
\newcommand\gb{\beta}
\newcommand\gd{\delta}
\newcommand\gD{\Delta}
\newcommand\gf{\varphi}
\newcommand\gam{\gamma}
\newcommand\gG{\Gamma}

\newcommand\gl{\lambda}

\newcommand\gss{\sigma^2}
\newcommand\eps{\varepsilon}

\newcommand\cB{\mathcal B}
\newcommand\cC{\mathcal C}

\newcommand\cF{\mathcal F}
\newcommand\cL{{\mathcal L}}
\newcommand\cO{\mathcal O}
\newcommand\cP{\mathcal P}
\newcommand\cS{{\mathcal S}}
\newcommand\cT{{\mathcal T}}
\newcommand\cU{{\mathcal U}}
\newcommand\cW{\mathcal W}

\newcommand\LT{\mathcal {LT}}

\newcommand\ct{\mathcal T}
\newcommand\ctb{\mathcal T''}

\newcommand\bS{{\overline S}}

\newcommand\tS{{\widetilde S}}
\newcommand\tW{{\widetilde W}}

\newcommand\ett[1]{\boldsymbol1[#1]} 
\newcommand\etta{\boldsymbol1} 
\def\[#1]{[\![#1]\!]}

\newcommand\qq{^{1/2}}
\newcommand\qqw{^{-1/2}}
\newcommand\qw{^{-1}}

\renewcommand{\=}{:=}

\newcommand\intoi{\int_0^1}

\newcommand\oi{[0,1]}
\newcommand\oii{[0,1]^2}

\newcommand\dd{\,\textup{d}}

\newcommand\half{\frac12}

\newcommand{\Lovasz}{Lov\'asz}
\newcommand{\Lovaszetal}{Lov\'asz and Szegedy}
\newcommand{\Borgsetal}{Borgs, Chayes, Lov\'asz, S\'os and Vesztergombi}
\newcommand{\tind}{t_{\mathrm{ind}}}

\newcommand{\cuq}{\overline{\cU}}

\newcommand{\cuoo}{\cU_\infty}

\newcommand{\ctq}{\overline{\ct}}
\newcommand{\ctoo}{\cT_\infty}
\newcommand{\ctbq}{\overline{\ctb}}
\newcommand{\ctboo}{\ctb_{\infty,\infty}}

\newcommand{\exch}{exchangeable}

\newcommand{\BUxx}[2]{\cB_{#1#2}}

\newcommand{\cbq}{\overline{\cB}}

\newcommand{\cboo}{\BUxx{\infty}{\infty}}

\newcommand{\cw}{\cW}
\newcommand{\cws}{\cW_{\mathsf s}}

\newcommand{\gwx}[1]{G(#1,W)}
\newcommand{\gwoo}{\gwx\infty}
\newcommand{\gwn}{\gwx{n}}

\newcommand{\gxwx}[2]{G(#1,#2)}

\newcommand{\Gnn}{G_{n_1,n_2}}
\newcommand{\uw}{\gG_W}
\newcommand{\ggw}{\gG_W}
\newcommand{\uww}{\gG''_W}
\newcommand{\ggww}{\gG''_W}
\newcommand{\tg}{threshold graph}
\newcommand{\btg}{bipartite threshold graph}
\newcommand{\fmu}{F_\mu}
\newcommand{\fmui}{F_\mu\qw}
\newcommand{\fmusss}{F_{\mu\sss}}
\newcommand{\smu}{S_\mu}
\newcommand{\wmu}{W_\mu}
\newcommand{\ggmu}{\gG_\mu}
\newcommand{\ggb}{\Gamma''}
\newcommand{\ggmub}{\ggb_\mu}

\newcommand{\dgg}[1]{D_{G;#1}}
\newcommand{\cps}{\cP_{\mathsf s}}
\newcommand{\oix}{$0$--$1$}
\newcommand{\inre}{^\circ}
\newcommand{\sss}{^\dagger}
\newcommand{\sssx}{\dagger} 
\newcommand{\tnxp}[1]{T_{#1,p}}
\newcommand{\tnp}{\tnxp{n}}
\newcommand{\tnnpp}{T_{n_1,n_2;p_1,p_2}}
\newcommand{\tnpx}{\widehat T_{n,p}}
\newcommand{\tnxx}[1]{\widehat T_{n,#1}}
\newcommand{\txtn}{T_{n;X,\taux}}
\newcommand{\txytnn}{T_{n_1,n_2;X,Y,\taux}}
\newcommand{\txxn}[1]{T_{n;X,#1}}
\newcommand{\txxxn}[1]{T_{n;#1}}
\newcommand{\tsxn}[1]{T_{n;#1}}
\newcommand{\tsxnn}[1]{T_{n_1,n_2;#1}}
\newcommand{\tsn}{\tsxn{S}}
\newcommand{\tsnn}{\tsxnn{S}}

\newcommand{\tnnx}[1]{T_{n_1,n_2;#1}}
\newcommand{\tnn}{T_{n_1,n_2}}
\newcommand{\ctnn}{\cT_{n_1,n_2}}

\newcommand{\cbc}{\cC_{\cB}}
\newcommand{\cbo}{\cO_{\cB}}
\newcommand{\cbw}{\cW_{\cB}}
\newcommand{\cbz}{\widetilde\cB}
\newcommand{\ctc}{\cC_{\cT}}
\newcommand{\cto}{\cO_{\cT}}
\newcommand{\ctw}{\cW_{\cT}}
\newcommand{\ctz}{\ct'}
\newcommand{\ixbp}{\iota_{\cB\cP}}
\newcommand{\ixpc}{\iota_{\cP\cC}}
\newcommand{\ixco}{\iota_{\cC\cO}}
\newcommand{\ixcw}{\iota_{\cC\cW}}
\newcommand{\ixow}{\iota_{\cO\cW}}
\newcommand{\ixwb}{\iota_{\cW\cB}}
\newcommand{\ixpb}{\iota_{\cP\cB}}
\newcommand{\ixtp}{\iota_{\cT\cP}}
\newcommand{\ixwt}{\iota_{\cW\cT}}
\newcommand{\gdc}{\gd_{\square}}
\newcommand{\gdcb}{\gd''_{\square}}
\newcommand{\pib}{\pi}
\newcommand{\pix}{\pi^*}
\newcommand{\wgx}{W^*(G)}
\newcommand{\wgnx}{W^*(G_n)}
\newcommand{\yy}{^*}
\newcommand{\elm}{\tau}
\newcommand{\BB}{B^*}
\newcommand{\tl}{T^L}
\newcommand{\tln}{\tl_n}
\newcommand{\taux}{t}
\newcommand{\nun}{_{\elm(n)}}
\newcommand{\nui}{_{\elm(i)}}
\newcommand{\tgn}{\widetilde G_n}
\newcommand{\so}{S^O}
\newcommand{\se}{S^E}
\newcommand{\mupp}{\mu_{p_1,p_2}}
\newcommand{\muppz}{\mu_{p_2,p_1}}
\newcommand{\spp}{S_{p_1,p_2}}
\newcommand{\ggbpp}{\ggb_{p_1,p_2}}
\newcommand{\nul}{\nu_L}
\newcommand{\gamd}{\gam_d}
\newcommand{\gamoo}{\gam_\infty}
\newcommand{\hS}{\hat S}
\newcommand{\hu}{\hat u}
\newcommand{\Onqqw}{O\bigpar{n\qqw}}

\newcommand{\Levy}{L\'evy}

\newcommand{\vanish}[1]{}

\newcommand\REM[1]{{\raggedright\texttt{[#1]}\par\marginal{XXX}}}

\newenvironment{comment}{\setbox0=\vbox\bgroup}{\egroup} 

\newcommand\urladdrx[1]{{\urladdr{\def~{{\tiny$\sim$}}#1}}}

\begin{document}
\title[Threshold graph limits]
{Threshold graph limits and random threshold graphs}

\date{August 12, 2009} 

\author{Persi Diaconis}
\address{Stanford and CNRS, Universit\'e de Nice Sophia Antipolis\\
Department of Mathematics\\
Stanford, CA 94305, USA}
\email{diaconis@math.stanford.edu}
\urladdrx{http://www-stat.stanford.edu/~CGATES/persi}

\author{Susan Holmes}
\address{Stanford \\
Department of Statistics\\
Stanford, CA 94305, USA}
\email{susan@stat.stanford.edu}
\urladdrx{http://www-stat.stanford.edu/~susan/}

\author{Svante Janson}
\address{Department of Mathematics, Uppsala University, PO Box 480,
SE-751~06 Uppsala, Sweden}
\email{svante.janson@math.uu.se}
\urladdrx{http://www.math.uu.se/~svante/}

\subjclass[2000]{} 

\begin{abstract} 
We study the limit theory of large threshold graphs and apply this to a variety of models
for random threshold graphs. The results give a nice set of examples for the emerging theory of
graph limits.
\end{abstract}

\maketitle

\section{Introduction}\label{S:intro}
\subsection*{Threshold Graphs}
Graphs have important applications in modern systems biology and
social sciences.
Edges are created between interacting genes or people who know each other.
However graphs are not objects which are naturally amenable to simple
statistical analyses, there is no natural average graph for instance.
Being able to predict or replace a graph by hidden (statisticians call them latent)
 real variables has
many advantages. This paper
studies such a class of graphs, that sits within the larger class of
interval graphs \cite{McMorris}, itself a subset of intersection
graphs \cite{Erdos}; see also \cite{Brandstadt}. 

Consider the following  properties of a simple graph $G$ on
$[n]\=\{1,2,\ldots,n \}$. 
\begin{eqenumerate}
\item\label{D1.1} 
There are real weights $w_i$ and a threshold value $t$ such that there is
an edge from $i$ to $j$ if and only if
$w_i+w_j>t$.
Thus ``{\em the rich people always know each other}".
\item\label{D1.2} 
$G$ can be built sequentially from the empty graph by adding vertices one at a time,
where each new vertex, is either isolated (non-adjacent to all the previous)
or dominating (connected to all the previous).
\item\label{D1.3} 
The graph is uniquely determined (as a labeled graph) by its degree sequence.
\item\label{D1.4}  
Any induced subgraph has either an isolated or a dominating vertex.
\item\label{D1induced}  
There is no induced subgraph $2K_2$, $P_4$ or
  $C_4$. 
(Equivalently, there is no \emph{alternating 4-cycle},
\ie, four distinct vertices $x,y,z,w$ with edges $xy$
and $zw$ but no edges $yz$ and $xw$; the diagonals
$xz$ and $yw$ may or may not exist.)
\end{eqenumerate}
These properties are equivalent and define the class of \emph{\tg{s}}.
The book by Mahadev and Peled \cite{MP} contains proofs
and several other seemingly different characterizations. 
Note that the complement of a threshold graph is a threshold graph
(by any of \ref{D1.1}--\ref{D1induced}).
By \ref{D1.2}, a threshold graph is either connected (if the
last vertex is dominating) or has an isolated vertex (if the last
vertex is isolated); clearly these two possibilities exclude each
other when $n>1$.

\begin{example}\label{Ex1}
  \begin{figure}[hbpt]
\includegraphics[width=3in,trim = 10mm 20mm 10mm  20mm,clip]{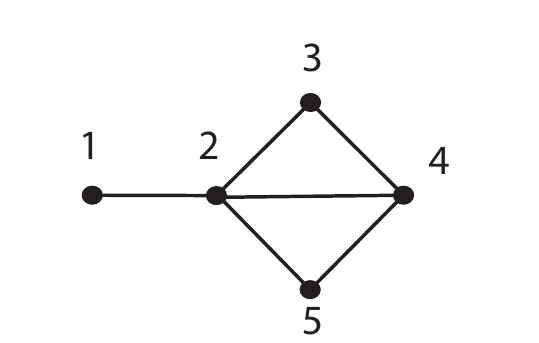}
\caption {A threshold graph}\label{FigEx1}
  \end{figure}
The graph in \refF{FigEx1}
is a threshold graph, from \ref{D1.1} by taking weights
1,5,2,3,2 on vertices 1--5 with $t=4.5$, or from \ref{D1.2} by adding
vertices 3, 5 (isolated), 4 (dominating), 1 (isolated) and 2 (dominating).  
\end{example}

While many familiar graphs are threshold graphs (stars or complete graphs for
example), many are not 
(e.g.\ paths or cycles of length 4 or more). 
For  example, of the $64$ labeled graphs on 4 vertices, 46 are
threshold graphs;
the other 18 are paths $P_4$, cycles $C_4$, and
pairs of edges $2K_2$ (which is the complement of $C_4)$.
Considering unlabeled graphs, there are 11 graphs
on 4 vertices, and 8 of them are threshold graphs.

\subsection*{Random Threshold Graphs}
It is natural to study random threshold graphs. 
There are several different natural random constructions; we will
in particular consider the following three:
\begin{eqenumerate}
\item\label{DRG1}  
From \ref{D1.1} by choosing $\{w_i\}_{1\leq i\leq n}$ as 
independent and identically distributed (\iid{}) random
variables from some probability distribution. (We also choose some
fixed $t$; we may assume $t=0$ by replacing $w_i$ by $w_i-t/2$.)
\item\label{DRG2}  
From \ref{D1.2} by 
ordering the vertices randomly and adding the
vertices one by one, each time choosing at random between the
qualifiers `dominating' or `isolated' with probabilities $p_i$ and
$1-p_i$, respectively, ${1\leq i\leq n}$. This is a simple random
attachment model in a similar  vein as those in \cite{Mitzenmacher}.  
We mainly consider the case when all $p_i$ are equal to a single
parameter $p\in\oi$.
\item\label{DRG3}
The uniform distribution on the set of threshold graphs.   
\end{eqenumerate}

\begin{figure}[htbp] 
\hbox{   \hskip-1.5cm
\includegraphics[width=3.5in]{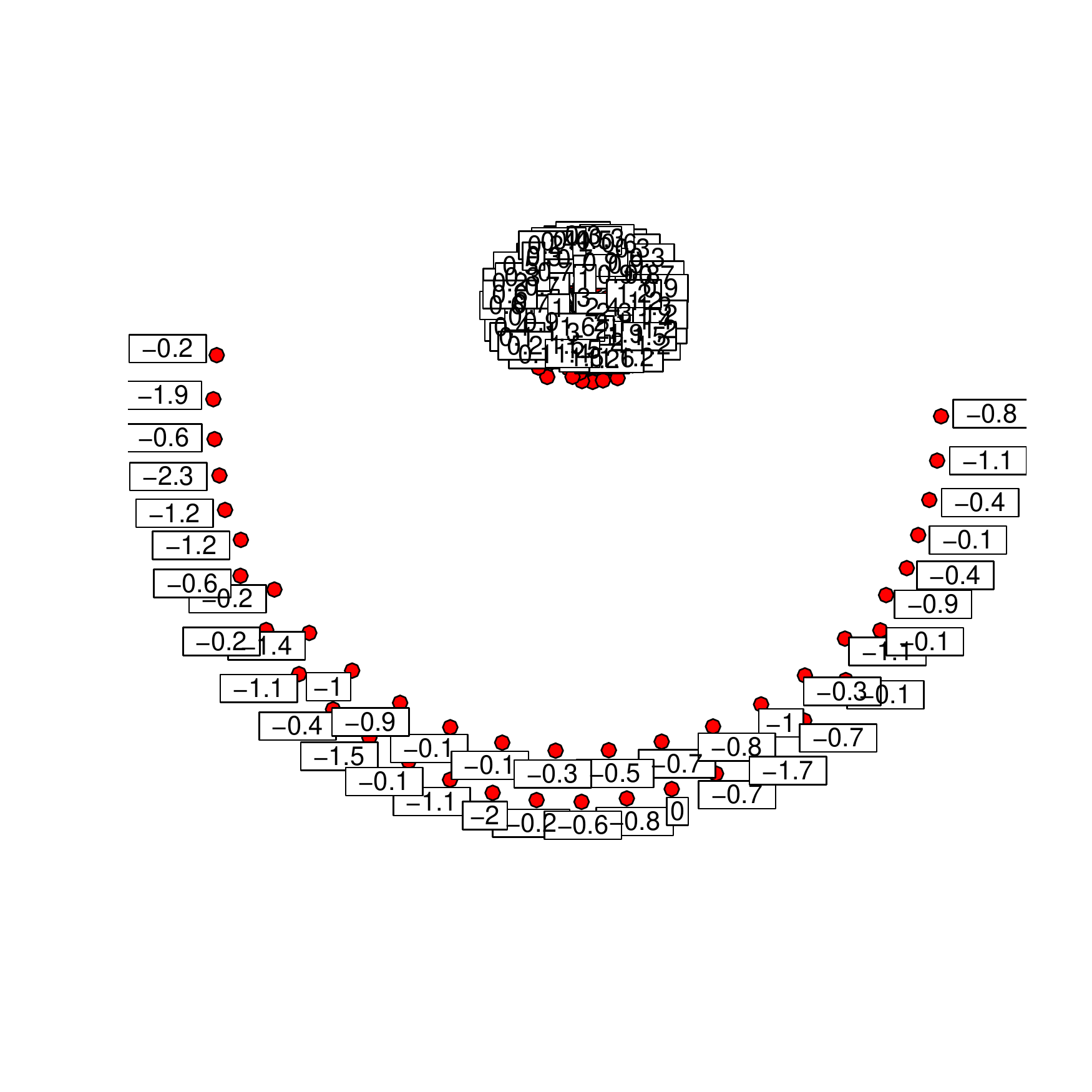}
\hskip-2cm
    \includegraphics[width=4.5in]{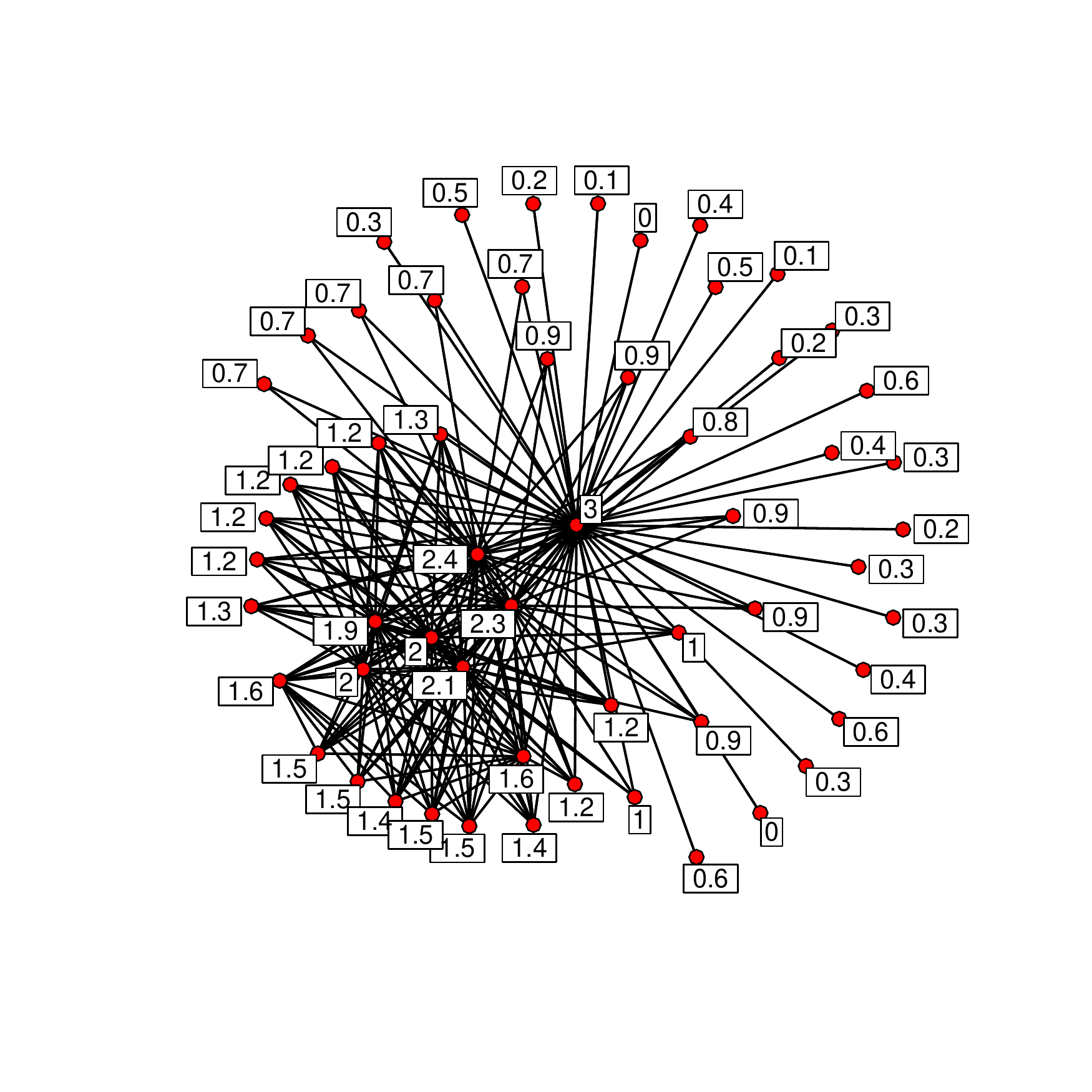}}
   \caption{A whole threshold graph with isolates (left) and with only
   the connected part expanded (right); the labels are the rounded
   weights $w_i$.} 
   \label{fig:exgr2}
\end{figure}

\begin{example}\label{Ex2}
Figure \ref{fig:exgr2} shows a random threshold graph 
constructed by \ref{DRG1}
with $w_i$
chosen independently 
from the standardized Normal distribution 
 and $t=3$. About half of the vertices are isolated, most of those with negative weights. 
\end{example}

 \begin{figure}[htbp] 
\hbox{\includegraphics[width=4.5in]{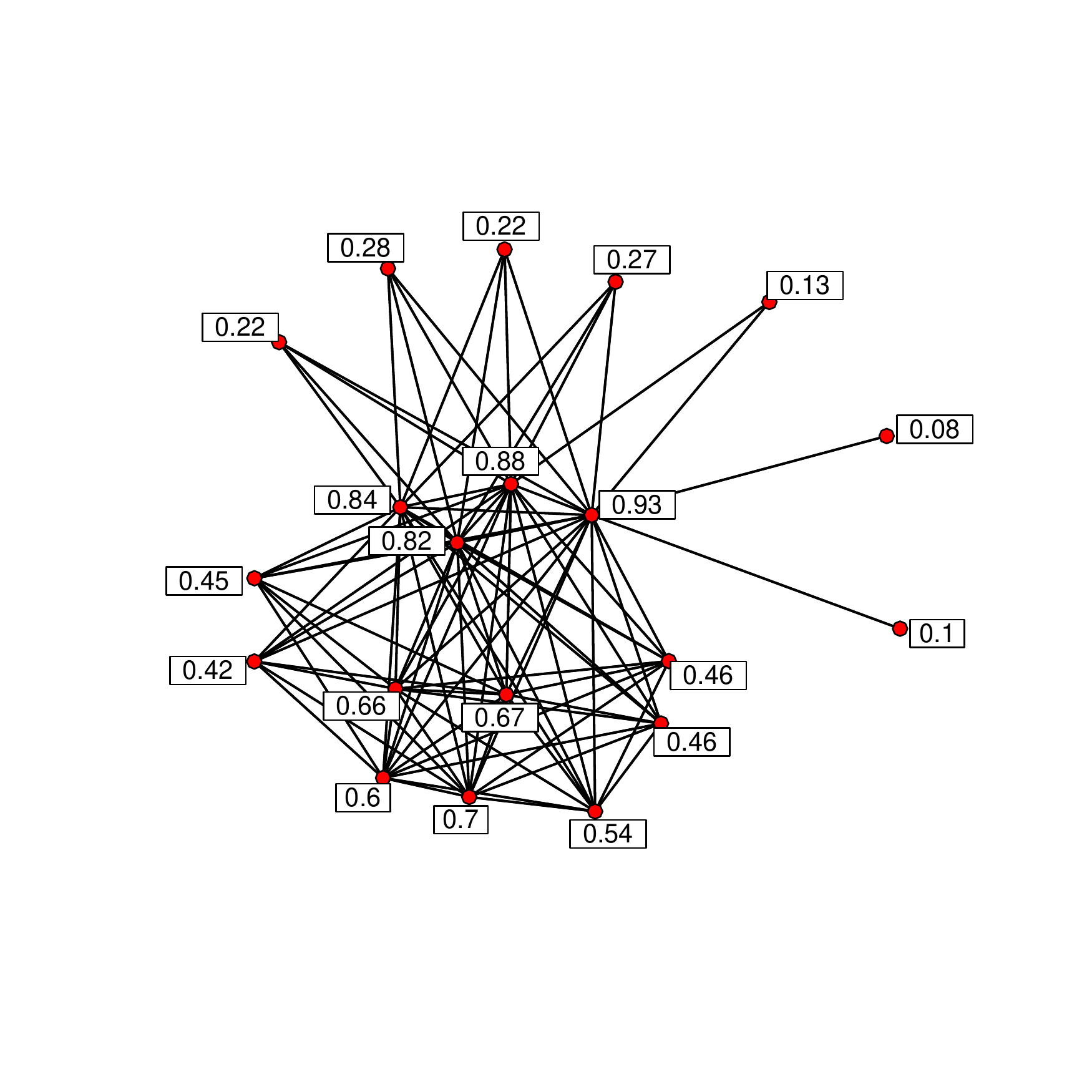}}
   \caption{A  threshold graph with $n=20$ and uniform $w_i$. It turns
   out that this instance had no
   isolates. The labels are the rounded weights $w_i$. }
   \label{fig:exgr3}
\end{figure}

 \begin{example}
   \label{Euniform}
 Figure \ref{fig:exgr3} shows a random threshold graph constructed by
 \ref{DRG1} with $w_i$ chosen as \iid{}
 uniform random variables on $[ 0,1 ]$ and $t=1$.
This instance is connected; this happens if and only if
the maximum and minimum of the $w_i$'s add to more than 1 (then there
 is a dominating vertex); in this example this has
probability $1/2$.

We show below (Corollaries \ref{C=} and \ref{C=2}) that this 
uniform weight model is
equivalent to adding isolated or dominating nodes as in
\ref{DRG2}
with probability
$p=1/2$, independently and in random order. 
It follows that
this same distribution
appears as the stationary distribution of a Markov chain on threshold graphs
which picks a vertex at random and changes it to dominating or isolated
with probability $1/2$ (this walk is analysed in \cite{BrownD}).
Furthermore, 
 it follows from \refSS{SSunlabeled} that these models yield a uniform
 distribution on the set of unlabeled \tg{s} of order $n$.
 \end{example}
 
 \subsection*{Bipartite Threshold Graphs}
 
We also study the parallel case of \btg{s} (difference graphs), both
for its own sake and because one of the main theorems is proved by
first considering the bipartite case.

By a \emph{bipartite graph}, we mean a graph with an explicit
bipartition of the vertex set; it can thus be written as $(V_1,V_2,E)$
where the edge set $E\subseteq V_1\times V_2$.
These following properties of a bipartite graph
are equivalent and define the class of \emph{\btg{s}}. (See \cite{MP}
for further characterizations.)
\begin{eqenumerate}
\item\label{B1.1} 
There are real weights $w_i'$, $i\in V_1$ and $w_j''$,
$j\in V_2$, and a threshold value $t$ such that there is
an edge from $i$ to $j$ if and only if
$w_i'+w_j''>t$.
\item\label{B1.2} 
$G$ can be built sequentially 
starting from $n_1$ white vertices
and $n_2$ black 
vertices in some fixed total order.
Proceeding in 
this order, make each white vertex dominate or isolated from 
all the black vertices that precede it and each black vertex
dominate or isolated from all earlier white vertices.
\item\label{B1.4}  
Any induced subgraph has either an isolated vertex or a vertex
dominating every vertex in the other part.
\item\label{B1induced}  
There is no induced subgraph $2K_2$.
\end{eqenumerate}

\begin{remarks}
1.
Threshold graphs were defined by 
Chv\'atal and Hammer \cite{CH77}.
Bipartite \tg{s} were studied by Hammer, Peled and Sun
\cite{HPS90} 
under the name
\emph{difference graphs} 
because they
can equivalently be characterized as the graphs $(V,E)$ for which there
exist weights $w_v$, $v\in V$, and a real number $t$
such that $|w_v|<t$ for every $v$ and $uv\in E \iff |w_u-w_v|>t$;
it is easily seen that every such graph is bipartite with
$V_1=\set{v:w_v\ge0}$ and $V_2=\set{v:w_v<0}$ and that is satisfies
the definition above (\eg, with $w'_v=w_v$ and $w''_v=-w_v$), and conversely.
We will use the name \btg{} to emphasize that we consider these graphs
equipped with a given bipartition. The same graphs were called
\emph{chain graphs} by 
Yannakakis \cite{Yan82} because each partition can be linearly ordered
for the inclusion of the neighborhoods of its elements.

2.
A suite of programs for working with threshold graphs appears
in
\cite{Hagberg06} 
with further developments 
in \cite{KMRS,MMK}.

3.
The most natural class of graphs built from a coordinate system are
commonly called
geometric  graphs \cite{Penrose} or geographical 
graphs \cite{KMRS,MMK}. Threshold graphs are
a special case of these. Their recognition and
manipulation in a statistical context
relies on useful measures on such graphs. We will start by defining
such measures and developing a limit theory.
\end{remarks}
\subsection*{Overview of the Paper}

The purpose of this paper is to study the limiting properties of large
threshold graphs in the spirit 
of the theory of graph limits developed by
Lov\'asz and Szegedy \cite{LSz}
and Borgs, Chayes, Lov\'asz,  S\'os, Vesztergombi \cite{BCL1} (and in
further papers by these authors and others).
As explained below, the limiting objects are not graphs,
but can rather be represented by symmetric functions $W(x,y)$ from
$[ 0,1 ]^2$ to $[ 0,1 ]$;
any sequence of graphs that converges in the appropriate way has such a limit. 
Conversely, 
such a function $W$ may be used to form a random 
graph $G_n$ by choosing independent random points $U_i$ in $[ 0,1 ]$,
and then for each pair $(i,j)$ with $1\leq i< j \leq n$ flipping
a biased coin with heads probability $W(U_i,U_j)$, putting an
edge from $i$ to $j$ if the coin comes up heads.
The resulting sequence of random graphs is (almost surely) an example
of a sequence of graphs converging to $W$.
For \refE{Euniform}, letting $\ntoo$, there is 
(as we show in greater generality in \refS{Srandom})
a limit $W$ that may be pictured as in \refF{fig:sps2}.

\begin{figure}[htbp] 
\includegraphics[width=2in]{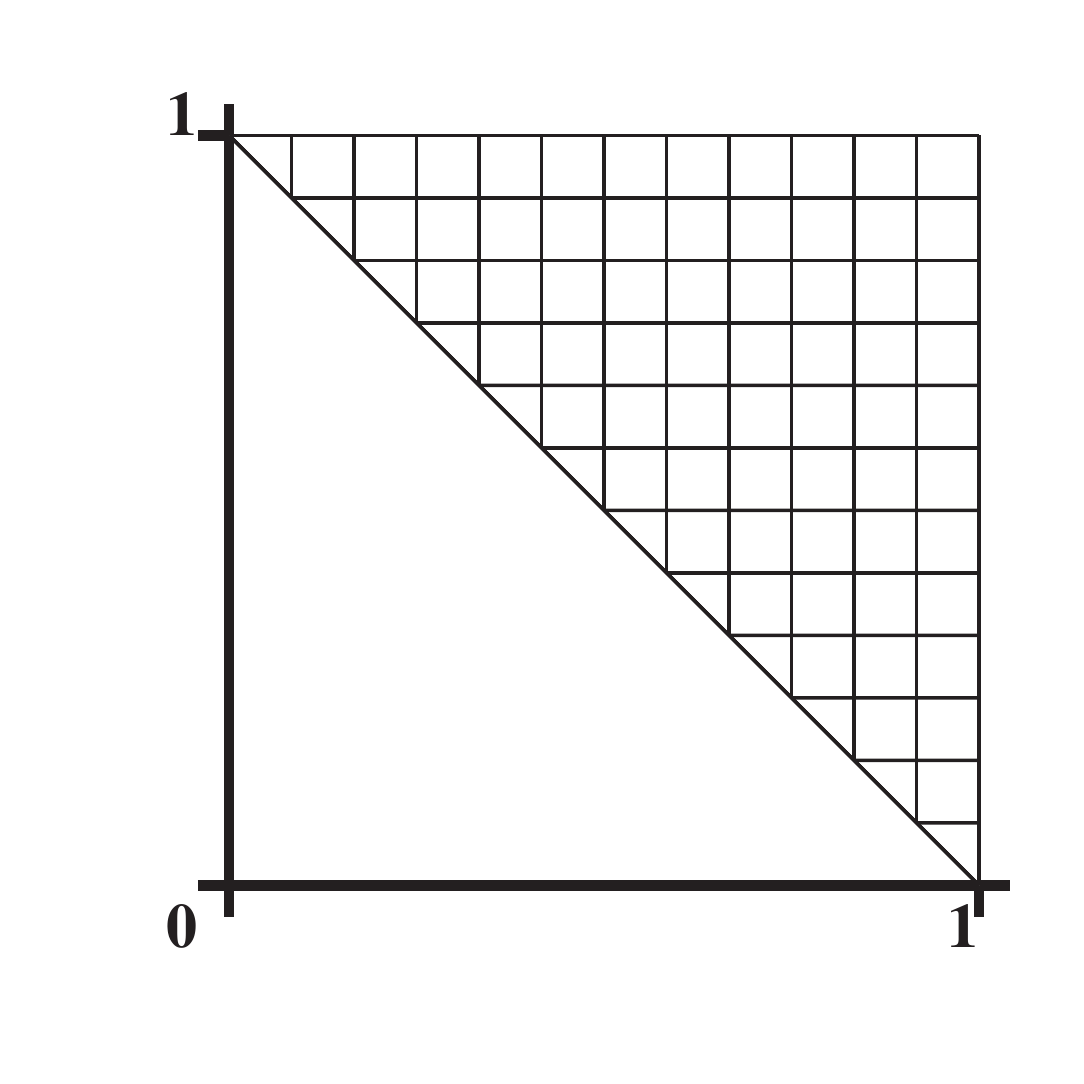}  
\caption{The function $W(x,y)$ for \refE{Euniform}. Hashed values have
$W(x,y)=1$, unhashed $W(x,y)=0$.
}
\label{fig:sps2}
\end{figure}

One of our main results (Theorems \ref{Too}) shows that graph limits of
threshold graphs have unique representations by increasing symmetric 
zero-one valued functions $W$. Furthermore, there is a
one-to-one correspondence 
between these limiting objects
and a certain type of
 `symmetric' probability distributions $P_W$ on $[ 0,1 ]$.
A threshold graphs is  characterized by its degree sequence;
normalizing this to be a probability distribution, say $\nu(G_n)$, we
show (\refT{Tlim}) that a sequence of  threshold graphs  
converges to $W$ when $n\to \infty$ if and only if 
$\nu(G_n)$ converges to $P_W$.
(Hence, $P_W$ can be regarded as the degree distribution of the limit.
The result that a limit of \tg{s} is determined by its degree distribution
is a natural analogue for the limit objects
of the fact that an unlabeled \tg{} is uniquely
determined by its degree distribution.)

 \begin{figure}[htbp] 
\hbox{       \includegraphics[width=4.5in]{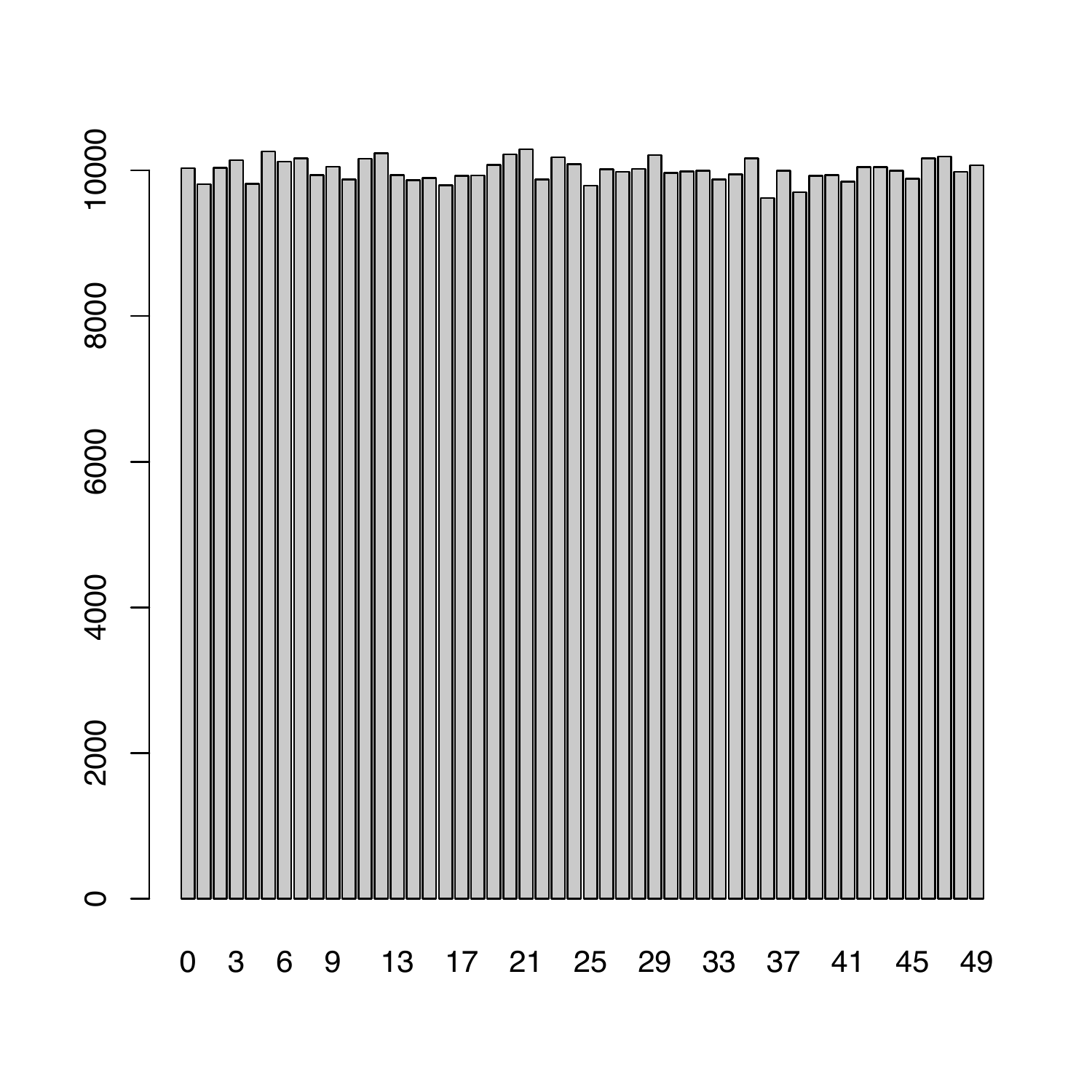}}
   \caption{Threshold graphs were generated with $n=50$ as in
	 \refE{Euniform}
with uniform 
$w_i$ and $t=1$; this is the degree histogram for a sample of
	 10,000 random graphs.}
   \label{fig:exgr4}
\end{figure}

 \refF{fig:exgr4} and \refF{fig:exgr5} 
show simulations 
of these results. In \refF{fig:exgr4},
10,000 graphs with $n=50$
were generated from
\ref{DRG1} with uniform weights as in \refE{Euniform}.

In the bipartite case,
there is a similar 1--1 
correspondence between the limit objects and probability
distributions on \oi; 
now all probability distributions on \oi{} appear in the representation
of the limits (\refT{Tboo}).

\refS{Suniform} discusses uniform random threshold graphs (both
labeled and unlabeled) and methods to generate them.
\refS{Sprel} gives a succint review of notation and graph limits. 
\refS{Sdeg} develops the limit theory of degree sequences; this is
not restricted to threshold graphs. \refS{Slim} develops the limit
theory  
for threshold graphs both deterministic and
random. \refS{Srandom} treats examples of random \tg{s} and
their limits, and \refS{SrandomB} gives corresponding examples and
results for random \btg{s}.
\refS{Sspectrum} treats the spectrum of the Laplacian of \tg{s}.

We denote the vertex and edge sets of a graph $G$ by $V(G)$ and
$E(G)$, and the numbers of vertices and edges by $v(G)\=|V(G)|$ 
and $e(G)\=|E(G)|$. For a bipartite graph we similarly use $V_j(G)$
and $v_j(G)$, $j=1,2$.

Throughout the paper, `increasing' and `decreasing' should be
interpreted in the weak sense (non-decreasing and non-increasing).
Unspecified limits are as \ntoo.


\section{Generating threshold graphs uniformly}\label{Suniform}

This section gives algorithms for generating uniformly distributed threshold graphs.
Both in the labeled case and in the unlabeled case. The algorithms are
used here for simulation and in Sections \refand{Srandom}{Sdegrees} 
to prove limit theorems. 

Let $\cT_n$ and $\LT_n$ be the sets of unlabeled and labeled
threshold graphs on $n$ vertices. These are different objects,
$\cT_n$ is a quotient of $\LT_n$, and we treat counting
and uniform generation separately for the two cases.
We assume in this section that $n\ge2$.

\subsection{Unlabeled threshold graphs}\label{SSunlabeled}
We can code an unlabeled threshold graph on $n$ vertices 
by a binary code $\ga_2\dotsm\ga_{n}$ of length $n-1$:
Given a code $\ga_2\dotsm\ga_n$, we construct
$G$ by \ref{D1.2} adding vertex $i$ as a
dominating vertex if and only if $\ga_i=1$
($i\ge2$).
Conversely, given $G$ of order $n\ge2$, let
$\ga_n=1$ if there is a dominating vertex ($G$
is connected)
and $\ga_n=0$ if there is an isolated  vertex ($G$ is disconnected);
we then remove one such dominating or isolated vertex and continue
recursively to define $\ga_{n-1},\dots,\ga_2$.

Since all dominating (isolated) vertices are equivalent to each other,
this coding gives a bijection between $\cT_n$ and
$\set{0,1}^{n-1}$. In particular,
\begin{equation*}
  |\cT_n|=2^{n-1},
\qquad n\ge1.
\end{equation*}
See \refF{FT3} for an example.

\begin{figure}
\includegraphics[width=3in]{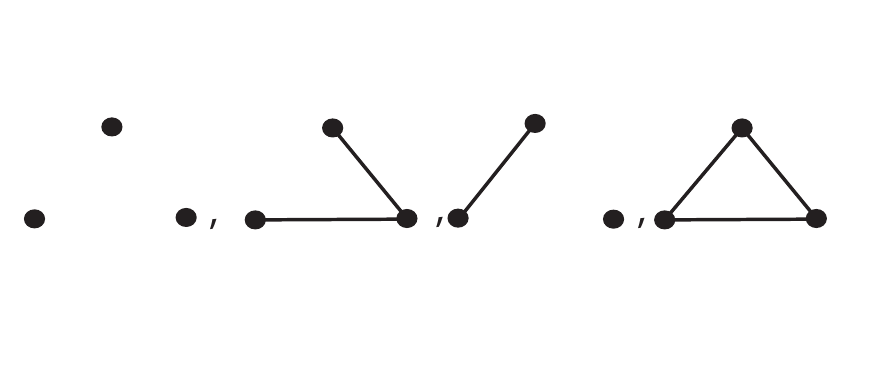}\\
\vskip-5mm
{\hskip 4mm 00 \hskip16mm 01\hskip16mm 10\hskip16mm 11}
\caption{The four graphs in $\cT_3$ and their codes.}
\label{FT3}  
\end{figure}

This leads to a simple algorithm to generate a uniformly distributed
random unlabeled \tg: we construct a random code by making $n-1$
coin flips. In other words:
\begin{Alg} \label{Alg1}
\emph{Algorithm for generating uniform random unlabeled threshold graphs
  of a given order $n$.}
\begin{description} 
\item[ Step 1] Add $n$ vertices by \ref{D1.2}, each
  time randomly choosing `isolated' or `dominating' with probability
  $1/2$.
\end{description}
\end{Alg}
This is thus the same as the second method in
\refE{Euniform}, so \refC{C=} shows that the first
method in \refE{Euniform} also yields uniform random
unlabeled threshold graphs (if we forget the labels).

The following notation is used to define two further algorithms
(\refSS{SSblock}) and for proof of the limiting results in \refS{Sdegrees}.

Define the \emph{extended binary code} of a \tg{} 
to be the binary code with the first binary digit repeated; it is thus
$\ga_1\ga_2\ga_3\dotsm\ga_n$ with
$\ga_1\=\ga_2$. The runs of 0's and 1's in the extended
binary code then correspond to blocks of vertices that can be added
together in \ref{D1.2} as either isolated or dominating vertices,
with the blocks alternating between isolated and dominating.
The vertices in each block are
equivalent and have, in particular, the same vertex degrees, while
vertices in different blocks can be seen to have different degrees. 
(The degree increases strictly from one dominating block to the next
and decreases strictly from one isolated block to the next, with every
dominating block having higher degree than every isolated block; \cf{}
\refE{Egen} below.)
The number of different vertex degrees thus equals the number of
blocks.

If the lengths of the blocks are $b_1,b_2,\dots,b_\elm$, then the number
of automorphisms of $G$ is thus $\prod_{j=1}^\elm b_j!$, since
the vertices in each block may be permuted arbitrarily.

Note that if
$b_1,\dots,b_\elm$ are the lengths of the blocks 
then
\begin{equation}\label{block}
  b_1\ge2,\qquad   b_k\ge1
  \text{ ($k\ge2$)},
\qquad \sum_{k=1}^\elm b_k=n.
\end{equation}
Since the blocks are alternatingly dominating or isolated, 
and the first block may be either, each
sequence $b_1,\dots,b_\elm$ satisfying \eqref{block}
corresponds to exactly 2 unlabeled \tg{s} of order
$n$. (These graphs are the complements of each other. One has
isolated blocks where the other has dominating blocks.)

\subsection{Labeled threshold graphs}\label{SSlabeled}

The situation is different for labeled threshold graphs.
For example, all of the $2^{\binom{3}{2}}=8$ labeled graphs with
$n=3$ turn out to be threshold graphs and  
for instance
$$ 
\includegraphics[width=3in]{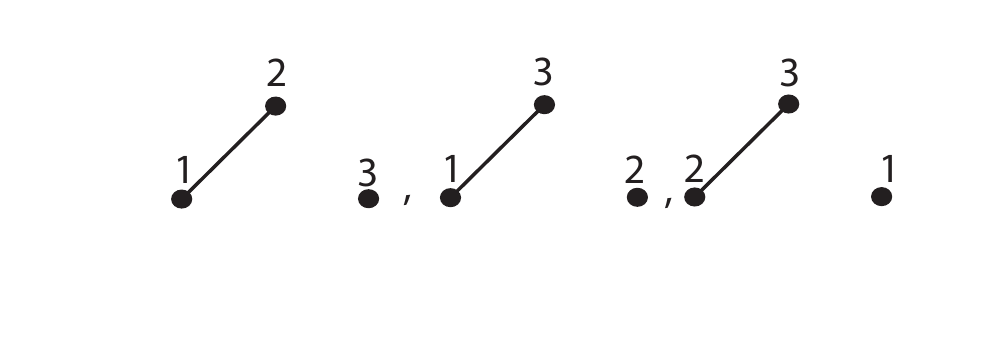}
$$
are distinguished. Hence the distribution of a uniform random labeled
threshold graph differs from the distribution of a uniform unlabeled
threshold graph (even if we forget the labels). In particular,
\refE{Euniform} does \emph{not} produce uniform
random labeled \tg{s}.

Let $G$ be an  unlabeled \tg{} with an extended code
having block lengths (runs) $b_1,\dots,b_\elm$.
Then the number of labeled \tg{s} corresponding to $G$ is
$n!/\prod_1^\elm b_j!$, since every such graph corresponds to a
unique assignment of the labels $1,\dots,n$ to the $\elm$
blocks, with $b_i$ labels to block $i$. (Alternatively and
equivalently, this follows from the number $\prod_1^\elm b_j!$
of automorphisms given above.)

The number $t(n)\=|\LT_n|$ of labeled
threshold graphs 
\cite[A005840]{OEIS}
has been studied by 
Beissinger and Peled
\cite{BPeled}.
Among other things, they show that
\begin{equation}
\sum_{n=0}^\infty t(n) \frac{x^n}{n!}=
\frac{(1-x)e^x}{2-e^x}
\end{equation}
so, by Taylor expansion,
$$
\begin{tabular}{r|rrrrrrrrrr}
$n$&1&2&3&4&5&6&7&8&9&10\\
\hline
$t(n)$&1&2&8&46&332&2874&29024&334982&4349492
& 62749906
\end{tabular}
$$
and by expanding the singularities (\cf{} \cite[Chapter IV]{FS})
the exact formula
\begin{equation}
\frac{t(n)}{n!}
=\sum_{k=-\infty}^\infty
\Bigpar{\frac{1}{\log 2+2\pi\ii
  k}-1}\Bigparfrac{1}{\log{2}+2\pi\ii k}^n,
\qquad n\ge2,
\end{equation}
where the leading term is the one with $k=0$, and thus the asymptotics
\begin{equation}\label{tn}
\frac{t(n)}{n!}=
\Bigpar{\frac{1}{\log 2}-1}\Bigparfrac{1}{\log{2}}^n+\epsilon(n),
\qquad |\epsilon(n)|\leq \frac{2\zeta(n)}{(2\pi)^n},
\end{equation}
where $\zeta(n)$ is the zeta function and thus $\zeta(n)\to1$;
furthermore,
\begin{equation}\label{rn}
t(n)=2R_n-2nR_{n-1},\ n\geq 2, 
\text{ with } 
R_n
=\sum_{k=1}^n k!S(n,k)
=\sum_{\ell=0}^\infty \frac{\ell^n}{2^{\ell+1}},
\end{equation}
where $S(n,k)$ are Stirling numbers;
$R_n$ is  the number of preferential arrangements of $n$ labeled
elements, or number of weak orders on $n$ labeled elements
\cite[A000670]{OEIS}, also called surjection numbers \cite[II.3]{FS}. 
(This is easily seen using the blocks above; the number of labeled
threshold graphs with a given sequence of blocks is twice (since the
first block may be either isolated or dominating) the number of
preferential arrangements with the same block sizes; if we did not
require $b_1\ge2$, this would yield $2R_n$, but we have
to subtract twice the number of preferential arrangements with $b_1=1$,
which is $2nR_{n-1}$.)
We note for future use the generating function \cite[(II.15)]{FS}
\begin{equation}
  \label{rgen}
\sum_{n=0}^\infty R_n\frac{x^n}{n!}=\frac1{2-e^x}.
\end{equation}

Let $t(n,j)$ be the number of labeled threshold graphs with $j$
isolated points. 
Then, as also shown in \cite{BPeled} (and easily seen),
for $ n\geq 2$,
\begin{align}
t(n,0)&=t(n)/2, \nonumber\\ 
\label {tnj}
t(n,j)&=
\begin{cases}
\binom{n}{j} t(n-j,0)=\frac12\binom nj t(n-j), 
& 0\leq j\leq n-2,
\\  
0, & j=n-1,
\\
1, & j=n.
\end{cases}
\end{align}
Thus  knowledge of 
$t(n)$ provides  $t(n,j)$.

These ingredients allow us to give an algorithm for
choosing uniformly in $\LT_n$.
\begin{Alg} \label{Alg2}
\emph{Algorithm for generating uniform random labeled threshold graphs
  of a given order $n$.}
\begin{description} 
\item[ Step 0] 
Make a list of $t(k)$ for $k$ between 1 and $n$. 
Make lists of $t(k,j)$ for $k=1,\dots,n$ and $j=0,\dots,k$.
\item[ Step 1] 
Choose an integer $j_0$ in $\{0,\dots,n\}$ with probability
that $j_0=j$ given by
$t(n,j)/t(n)$. Choose (at random) a subset of
$j_0$ points in \set{1,\dots,n}. These are the isolated vertices in the graph.
Let $n'\=n-j_0$ be the number of remaining points.
If $n'=0$ then stop.
\item[ Step 2] 
 Choose an integer $j_1$ in $\{1,\dots,n'\}$ with probability
that $j_1=j$ given by
$t(n',j)/(t(n')-t(n',0))=2t(n',j)/t(n')$ and choose (at
 random) $j_1$ points of those remaining; 
these will dominate all further points, so add edges between these
 vertices and from them to all remaining points.
Update $n'$ to $n'-j_1$, the number of remaining points.
If $n'=0$ then stop.
\item[ Step 3] 
 Choose an integer $j_2$ in $\{1,\dots,n'\}$ with probability
that $j_2=j$ given by
$2t(n',j)/t(n')$ and choose (at
 random) $j_2$ points of those remaining; 
these will be isolated among the remaining points, so no further edges
are added.
Update $n'$ to $n'-j_1$, the number of remaining points.
If $n'=0$ then stop.
\item[Step 4] 
Repeat from Step 2 with the remaining $n'$ points.
\end{description}
\end{Alg}

Alternatively, instead of selecting the subsets in Steps 1 and 2 at
random, we may choose them in any way, provided the algorithm begins
or ends with a random permutation of the points.

The algorithm works because of a characterization of threshold
graphs by Chv\'atal and Hammer \cite{CH77}, \cf{}  \ref{D1.4}: 
A graph is a threshold graph iff any subset
$S$ of vertices contains at least one isolate or one dominating vertex
(within the graph induced by $S$).
Thus in step 2, since there are no
isolates among the $n'$ vertices left there must be at least one dominating
vertex.
(Note that $j_0$ may be zero, but not $j_1,j_2\dots$.)
The probability distribution for the number of dominating vertices follows
the same law as that of the isolates 
because the complement of a \tg{} is a \tg{}
(or because of the interchangeability
of 0's and 1's in the binary coding given earlier in this
section).

Note that this algorithm treats vertices in the reverse of the order in
\ref{D1.2} where we add vertices instead of peeling them off
as here. It follows that we obtain the extended binary code of the
graph by taking runs of $j_0$ 0's, $j_1$ 1's,  $j_2$
0's, and so on, and then reversing the order. Hence, in the notation
used above, 
the sequence $(b_k)$ equals $(j_{k})$ in reverse order, ignoring
$j_0$ if $j_0=0$.
(In particular
note that the last $j_k\ge2$, since $t(n',n'-1)=0$ for
$n'\ge2$, which corresponds to the first block $b_1\ge2$.)

\begin{example}\label{Egen}
A sequence of $j$s generated for a threshold graph of size 20 is
{\tt 0 2 3 1  1 1 3  1 1 3 1 1  2}, which yields the sequence
{\tt d d i i i d i  d   i i  i  d   i   d   d   d  i  d  i i}
of dominating and isolated vertices. 
A random permutation of \set{1,\dots,20} was generated and we obtain
\begin{equation*}
\newcommand\ttx{\tt\!}
\begin{tabular}{rrrrrrrrrrrrrrrrrrrr}
\ttx13&\ttx2&\ttx11&\ttx15&\ttx8&\ttx20&\ttx6&\ttx12&\ttx16&\ttx4&\ttx18&\ttx7&\ttx10&\ttx9&\ttx14&\ttx17&\ttx
1&\ttx19&\ttx5&\ttx3\\
\ttx d&\ttx d&\ttx i&\ttx i&\ttx i&\ttx d&\ttx i&\ttx d&\ttx i&\ttx i&\ttx i&\ttx d&\ttx i&\ttx d&\ttx d&\ttx d&\ttx i&\ttx
d&\ttx i&\ttx i
\end{tabular}
\end{equation*}
where \texttt{d} signifies that the vertex is connected to all
later vertices in this list.
The degree sequence is thus, taking the vertices in this order:
19, 19, 2, 2, 2, 16, 3, 15, 4, 4, 4, 12, 5, 11, 11, 11, 8, 10,  9, 9.
The 
extended binary code 
$00101110100010100011$
is obtained by translating \texttt i to 0 and \texttt d to 1,
and reversing the order.
\end{example}

\begin{comment}
Note that the last $j$ is at least 2, since $t(n',n'-1)=0$
for every $n'\ge2$. Hence, the sequence of 
\texttt{d}'s and \texttt{i}'s always ends with at least two identical symbols.

Note that the vertex degrees are constant in each block of vertices
assigned in one of the steps, \ie, for each run of
\texttt{i}'s or \texttt{d}'s, and that they 
decrease from each run of \texttt{d}'s to the next and
increase from each run of \texttt{i}'s to the next, with each
vertex labeled \texttt{d} having a higher degree than every
vertex labeled with \texttt{i}; the number of different
vertex degrees is thus equal to the number of $j$'s chosen by
the algorithm, which equals the number of runs in the sequence of
\texttt{d}'s and \texttt{i}'s (or in the extended
binary code). 
\end{comment}

 \begin{figure}[htbp] 
\hbox{       \includegraphics[width=4.5in]{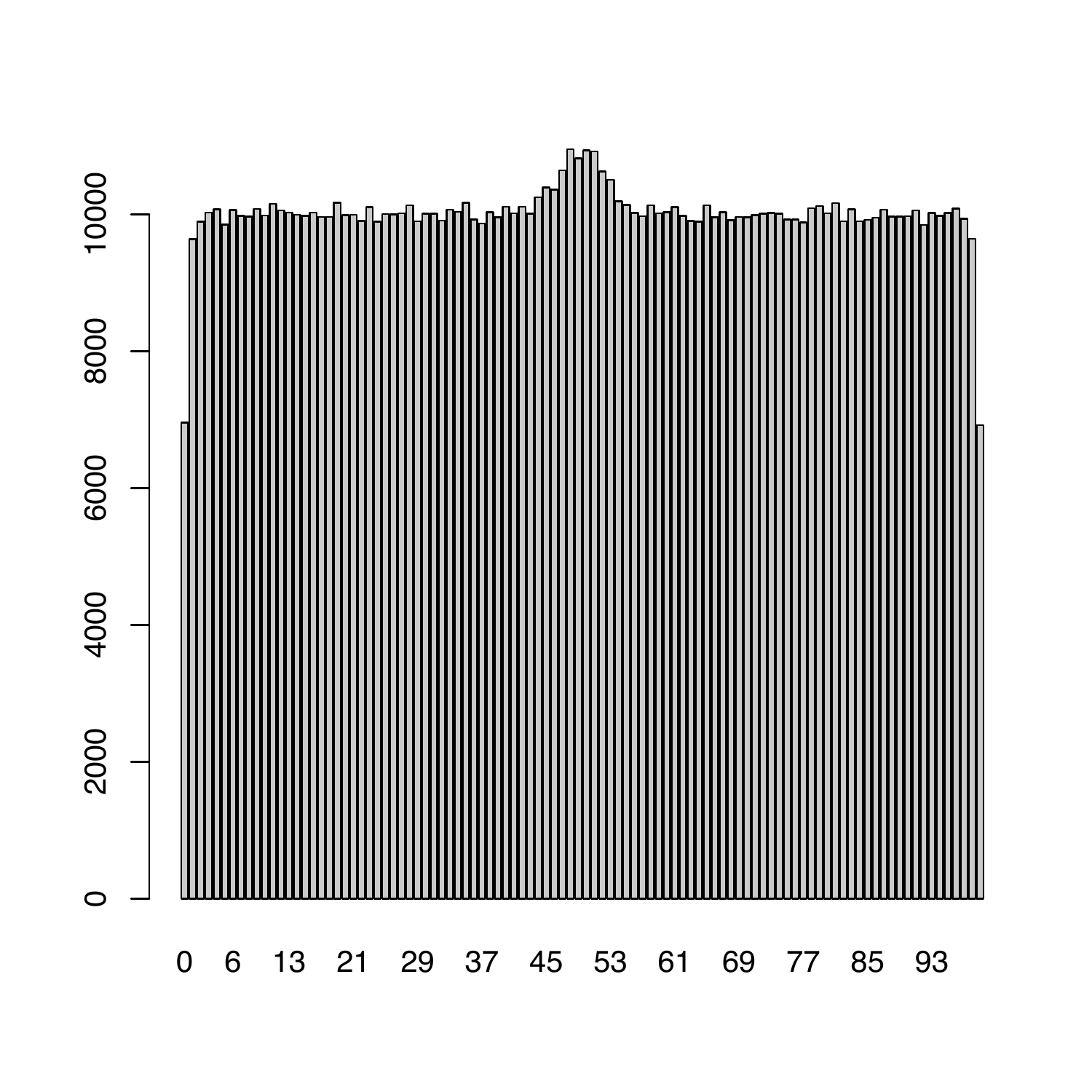}}
   \caption{Threshold graphs were generated according to the algorithm of this section,
   this is the degree histogram.}
   \label{fig:exgr5}
\end{figure}

In \refF{fig:exgr5},
10,000 graphs were generated with $n=100$ according to
the uniform distribution over all labeled threshold graphs.
We discuss the central `bump' and other features of \refF{fig:exgr5}
in Theorem \ref{Tend}.

\subsection{The distribution of block lengths}\label{SSblock}

We have seen in \refSS{SSunlabeled} that if
$b_1,\dots,b_\elm$ are the lengths of the blocks of isolated or
dominating vertices added to the graph when building it as in
\ref{D1.2}, then \eqref{block} holds.
Consider now a sequence of independent integer random variables
$B_1,B_2,\dots$ with $B_1\ge2$ and
$B_j\ge1$ for $j\ge2$, and let
$S_k\=\sum_{j=1}^k B_j$ be the partial sums. If some
$S_\elm=n$, then stop and output the sequence
$(B_1,\dots,B_\elm)$.
Conditioning on the event that $S_\elm=n$ for some $\elm$, this
 yields a random sequence $b_1,\dots,b_\elm$ satisfying
\eqref{block}, and the probability that we obtain a given
sequence $(b_j)_1^\elm$ equals
$c\prod_{j=1}^\elm\P(B_j=b_j)$ for some normalizing
constant $c$. We now specialize to the case when
$B_1\eqd(\BB\mid \BB\ge2)$ and $B_j\eqd(\BB\mid \BB\ge1)$ for $j\ge2$,
for some given random variable $\BB$. Then the (conditional) probability
of obtaining a given $b_1,\dots,b_\elm$ satisfying
\eqref{block} can be written
\begin{equation}\label{bp}
 c' \prod_{j=1}^\elm \frac{\P(\BB=b_j)}{\P(\BB\ge1)}
\end{equation}
(with $c'=c\P(\BB\ge1)/\P(\BB\ge2)$).

There are two important cases. First, if we take
$\BB\sim\Ge(1/2)$, then
$\P(\BB=b_j)/\P(\BB\ge1)=2^{-b_j}$, and
thus \eqref{bp} yields
$c'2^{-\sum_jb_j}=c'2^{-n}$, so the probability
is the same for all allowed sequences.
Hence, in this case the distribution of the constructed sequence
is uniform on
the set of sequences satisfying \eqref{block}, so it equals
the distribution of block lengths for a random unlabeled
\tg{} of size $n$.

The other case is $\BB\sim\Po(\log 2)$. Then
$\P(\BB\ge1)=1-e^{-\log2}=1/2$, and 
$\P(\BB=b_j)/\P(\BB\ge1)=(\log 2)^{b_j}/b_j!$.
Thus, \eqref{bp} yields the probability $c'(\log
2)^n/\prod_jb_j!$, which is proportional to the number
$2\cdot n!/\prod_jb_j!$ of labeled \tg{s} with the block
lengths $b_1,\dots,b_\elm$. 
Hence, in this case the distribution of the constructed sequence
equals
the distribution of block lengths for a random labeled
\tg{} of size $n$. 

We have shown the following result.

\begin{theorem}\label{Tblock}
  Construct a random sequence $B_1,\dots,B_\elm$ as above, based
  on a random variable $\BB$, stopping when 
$\sum_1^\elm B_j\ge n$ and conditioning on 
$\sum_1^\elm B_j= n$. 
  \begin{romenumerate}
	\item
If $\BB\sim\Ge(1/2)$, then $(B_1,\dots,B_\elm)$ has
the same distribution as the block lengths in a random unlabeled
threshold graph of order $n$.
	\item
If $\BB\sim\Po(\log2)$, then $(B_1,\dots,B_\elm)$ has
the same distribution as the block lengths in a random labeled
threshold graph of order $n$.
  \end{romenumerate}
\end{theorem}

It follows that the length of a typical (for example a random) block
converges in distribution to $(\BB\mid\BB\ge1)$.
\refT{Tblock} 
also leads to another algorithm to construct uniform random \tg{s}.

\begin{Alg} \label{Alg3}
\emph{Algorithm for generating uniform unlabeled or labeled \tg{s}
 of a given order $n$.}
 \begin{description}
	\item[Step 1]
In the unlabeled case, let $\BB\sim\Ge(1/2)$. In the
labeled case, let $\BB\sim\Po(\log2)$.
	\item[Step 2]
Choose independent random numbers $B_1,B_2,\dots,B_\elm$, with
$B_1\eqd(\BB\mid\BB\ge2)$  and
$B_j\eqd(\BB\mid\BB\ge1)$, $j\ge2$, until the sum $\sum_1^\elm B_j\ge n$.
	\item[Step 3]
If $\sum_1^\elm B_j> n$, start again with Step 2.
	\item[Step 4]
We have found $B_1,\dots,B_\elm$ with $\sum_1^\elm B_j=n$.
Toss a coin to decide whether the first block is isolated or dominating;
the following blocks alternate. Construct a threshold graph by adding
vertices as in \ref{D1.2}, block by block.
	\item[Step 5]
In the labeled case, make a random labeling of the graph.
  \end{description}
\end{Alg}

By standard renewal theory, the probability that $\sum_1^\elm B_j$ is
exactly $n$ is asymptotically
$1/\E(\BB\mid\BB\ge1)=\P(\BB\ge1)/\E\BB$, which is $1/2$ in
the unlabeled case and $1/(2\log2)\approx0.72$ in the
labeled case, so we do not have to do very many restarts in Step 3.

\section{Graph limits}\label{Sprel}

This section reviews needed tools from the emerging field of graph limits. 

\subsection{Graph limits}\label{SSgraphlim}
Here we review briefly
the theory of graph limits as described in
\Lovaszetal{} \cite{LSz},
\Borgsetal{} \cite{BCL1} and
Diaconis and Janson \cite{DJ}.

If $F$ and $G$ are two graphs, 
let $t(F,G)$ be the probability that a random mapping $\phi:V(F)\to V(G)$
defines a graph homomorphism, \ie, that $\phi(v)\phi(w)\in E(G)$ when
$vw\in E(F)$. (By a random mapping we mean a mapping uniformly chosen
among all $v(G)^{v(F)}$ possible ones; the images of the vertices in
$F$ are thus independent and uniformly distributed over $V(G)$, \ie,
they are obtained by random sampling with replacement.)

The basic definition is that a sequence $G_n$ of (generally
unlabeled) graphs
converges if $t(F,G_n)$ converges for every graph $F$;
as  in \cite{DJ} we will further assume
$v(G_n)\to\infty$.
More precisely, 
the (countable and discrete) set $\cU$
of all unlabeled graphs can be embedded in a compact metric space
$\cuq$ such that a sequence 
$G_n\in\cU$ of graphs with $v(G_n)\to\infty$
converges in $\cuq$ to some limit $\gG\in\cuq$ if and only if $t(F,G_n)$
converges for every graph $F$ (see \cite{LSz}, \cite{BCL1}, \cite{DJ}). 
Let $\cuoo\=\cuq\setminus \cU$ be the
set of proper limit elements; we call the elements of $\cuoo$
\emph{graph limits}.
The functionals $t(F,\cdot)$
extend to continuous functions on $\cuq$, so $G_n\to\gG\in\cuoo$ if and
only if $v(G_n)\to\infty$ and
$t(F,G_n)\to t(F,\gG)$ for every graph $F$.

Let $\cw$ be the set of all measurable functions $W:\oi^2\to\oi$ and
let $\cws$ be the subset of symmetric functions.
The main result of \Lovaszetal{} \cite{LSz} is that every element of
$\cuoo$ can be represented by a (non-unique) function $W\in\cws$.
We let $\gG_W\in\cuoo$ denote the graph limit defined by $W$.
(We sometimes use the notation $\gG(W)$ for readability.)
Then, for every graph $F$, 
\begin{equation}\label{tuw}
  t(F,\uw)
=
\int_{\oi^{v(F)}} \prod_{ij\in E(F)} W(x_i,x_j) \dd x_1\dotsm \dd x_{v(F)}.
\end{equation}
Moreover, define, for every $n\ge1$, a random graph
$\gwn$ as follows: first choose a sequence $X_1,X_2,\dots, X_n$ 
of \iid{} random variables uniformly distributed on $\oi$, and then,
given this sequence, 
for each pair $(i,j)$ with $i<j$ draw an edge 
$ij$ with probability $W(X_i,X_j)$, independently for
all pairs $(i,j)$ with $i<j$.
Then the random graph $\gwn$ converges to $\ggw$ \as{} as \ntoo.

If $G$ is a graph, with $V(G)=\set{1,\dots,v(G)}$ for simplicity,
we define a function $W_G\in\cws$ by partitioning
$\oi$ into $v(G)$ intervals $I_i$, $i=1,\dots,v(G)$, and letting $W_G$
be the indicator $\ett{ij\in E(G)}$ on $I_i\times I_j$. (In other
words, $W_G$ is a step function corresponding to the adjacency matrix
of $G$.) We let $\pi(G)\=\gG(W_G)$ denote the corresponding object in
$\cuoo$. 
It follows easily from \eqref{tuw} that $t(F,\pi(G))=t(F,G)$ for every
graph $F$. In particular, if $G_n$ is a sequence of graphs with
$v(G_n)\to\infty$, then $G_n$ converges to some graph limit $\gG$ if
and only if $\pi(G_n)\to\gG$ in $\cuoo$.
(Unlike \cite{LSz} and \cite{BCL1} we distinguish between
graphs and limit objects and we do not identify $G$ and $\pi(G)$, see
\cite{DJ}.) 

\subsection{Bipartite graphs and their limits}
In the bipartite case,
there are analoguous definitions and results
(see \cite{DJ} for further details). 
We define a bipartite graph to be a graph $G$ with an explicit
bipartition $V(G)=V_1(G)\cup V_2(G)$ of the vertex set, such
that the edge set $E(G)\subseteq V_1(G)\times V_2(G)$.
Then we define $t(F,G)$ in the same way as above but now for
bipartite graphs $F$, by
letting $\phi=(\phi_1,\phi_2)$ be a pair of random mappings
$\phi_j:V_j(F)\to V_j(G)$.
We let $\cB$ be the set of all unlabeled bipartite graphs
and embed $\cB$ in a compact metric space $\cbq$.
A sequence $(G_n)$ of bipartite graphs with $v_1(G_n),v_2(G_n)\to\infty$
converges in $\cbq$ if and only if $t(F,G_n)$ converges for every
bipartite graph $F$. Let $\cboo$ be the (compact) set of all such
limits; we call the elements of $\cboo$
\emph{bipartite graph limits}.
Every element of $\cboo$ can be represented by a (non-unique) function
$W\in\cW$. We let $\ggww\in\cboo$ denote the element represented by
$W$ and have, for every bipartite $F$
\begin{equation}\label{tuww}
  t(F,\uww)
=
\int_{\oi^{v_1(F)+v_2(F)}} \prod_{ij\in E(F)} W(x_i,y_j) 
  \dd x_1\dotsm \dd x_{v_1(F)}  \dd y_1\dotsm \dd y_{v_2(F)}.
\end{equation}

Given $W\in\cW$ and $n_1,n_2\ge1$, we define
a random bipartite graph $G(n_1,n_2,W)$ by an analogue of the
construction in \refSS{SSgraphlim}:
 first choose two sequences $X_1,X_2,\dots, X_{n_1}$  and
 $Y_1,Y_2,\dots, Y_{n_2}$  
of \iid{} random variables uniformly distributed on $\oi$, and then,
given thess sequences, 
for each pair $(i,j)$ draw an edge 
$ij$ with probability $W(X_i,Y_j)$, independently for
all pairs $(i,j)\in [n_1\times[n_2]$. 

If $G$ is a bipartite graph we define $W_G\in\cW$ similarly as
above (in general with 
different numbers of steps in the two variables; note that $W_G$
now in general is not symmetric) and let 
$\pib(G)\=\ggb(W_G)$. 
Then, 
by \eqref{tuww}, $t(F,\pi(G))=t(F,G)$ for every bipartite
graph $F$. Hence, if $G_n$ is a sequence of bipartite graphs with
$v_1(G_n),v_2(G_n)\to\infty$, then $G_n$ converges to some bipartite
graph limit $\gG$ if
and only if $\pi(G_n)\to\gG$ in $\cboo$.

\subsection{Cut-distance}
\Borgsetal{} \cite[Section 3.4]{BCL1} define a (pseudo-)metric $\gdc$
on $\cws$ called the \emph{cut-distance}. 
This is only a pseudo-metric since two different functions in $\cws$
may have cut-distance 0 (for example, if one is obtained by a measure
preserving transformation of the other, see further
\cite{BCL:unique} and \cite{DJ}), and it is shown 
in \cite{BCL1} that, in fact, 
$\gdc(W_1,W_2)=0$ if and only if $t(F,W_1)=t(F,W_2)$ for every graph
$F$, \ie, if and only if $\gG_{W_1}=\gG_{W_2}$ in $\cuoo$. Moreover,
the quotient space $\cws/\gdc$, 
where we identify elements of $\cws$ with cut-distance 0, 
is a
compact metric space and
the mapping $W\mapsto \ggw$ is a homeomorphism of $\cws/\gdc$ onto $\cuoo$.

This extends to the bipartite case. 
In this case, we define $\gdcb$ on
$\cW$ as $\gdc$ is defined in  \cite[Section 3.4]{BCL1}, but allowing
different measure 
preserving mappings for the two coordinates. 
Then,
if we identify elements in $\cW$ with cut-distance 0,
 $W\mapsto\ggww$ becomes a homeomorphism of $\cW/\gdcb$ 
onto $\cboo$.
Instead of repeating and modifying
the complicated proofs from \cite{BCL1}, one can use their result in
the symmetric case and define an embedding $W\mapsto\tW$
of $\cW$ into $\cws$ by
\begin{equation*}
  \tW(x,y)=
  \begin{cases}
0, & x<1/2,\, y<1/2;	
\\
1, & x>1/2,\, y>1/2;	
\\
\frac14+\frac12 W(2x-1,2y), & x>1/2,\, y<1/2;	
\\
\frac14+\frac12 W(2y-1,2x), & x<1/2,\, y>1/2.
  \end{cases}
\end{equation*}
It is easily seen that $\gdcb(W_1,W_2)$ and $\gdc(\tW_1,\tW_2)$ are
equal within some constant factors, for $W_1,W_2\in\cW$, and that for
each graph $F$, $t(F,\tW)$ is a linear combination of $t(F_i,W)$ for a
family of bipartite graphs $F$ (obtained by partitioning $V(F)$ and
erasing edges within the two parts).
This and the results in \cite{BCL1}, 
together with the simple fact that $W\mapsto t(F,W)$ is
continuous for $\gdcb$ for every bipartite graph $F$, imply easily
the result  claimed.
 
\subsection{A reflection involution}
If $G$ is a bipartite graph, let $G\sss$ be the graph obtained by
interchanging the order of the two vertex sets; thus,
$V_j(G\sss)=V_{3-j}(G)$ and $E(G\sss)=\set{uv:vu\in E(G)}$.
We say that $G\sss$ is the \emph{reflection} of
$G$.
Obviously, $t(F,G\sss)=t(F\sss,G)$ for any bipartite graphs $F$ and $G$. It
follows that if $G_n\to \gG\in\cbq$, then $G_n\sss\to\gG\sss$ for some
$\gG\sss\in\cbq$, and this defines a continuous map of $\cbq$ onto
itself which extends the map just defined for bipartite graphs.
We have, by continuity,
\begin{equation}\label{sss}
  t(F,\gG\sss)=t(F\sss,\gG),
\qquad F\in\cB,\, G\in\cbq.
\end{equation}
Furthermore, $\gG\sss{}\sss=\gG$, so the map is an involution, and it
maps $\cboo$ onto itself.

For a function $W$ on $\oii$, let $W\sss(x,y)\=W(y,x)$
be its reflection in the main diagonal. It follows from \eqref{tuww}
and \eqref{sss} that
$\ggb(W\sss)=\ggb(W)\sss$.

\subsection{Threshold graph limits}\label{SSthrlim}

Let $\cT\=\bigcup_{n=1}^\infty\cT_n$ be the
family of all (unlabeled) \tg{s}. Thus $\cT$ is a
subset of the family $\cU$ of all unlabeled graphs, and we
define $\ctq$ as the closure of $\cT$ in
$\cuq$, and $\ctoo\=\ctq\setminus\cT=\ctq\cap\cuoo$,
\ie, the set of proper limits of sequences of threshold graphs; we
call these \emph{\tg{} limits}.

In the bipartite case, we similarly consider the set 
$\ctb\=\bigcup_{n_1,n_2\ge1}\ctnn\subset\cB$ of all bipartite threshold
graphs, and let 
$\ctbq\subset\cbq$ be its closure in $\cbq$  and
$\ctboo\=\ctbq\cap\cboo$ the set 
of proper limits of sequences of \btg{s};
we call these \emph{\btg{} limits}.

Note that $\ctq$, $\ctoo$, $\ctbq$,
$\ctboo$ are compact metric spaces, since they are closed
subsets of $\cuq$ or $\cbq$. 

We will give concrete representations of 
the threshold graph limits in \refS{Slim}.
Here we only give a more abstract characterization.

Recall that $t(F,G)$ is defined as the proportion of maps
$V(F)\to V(G)$ that are graph homomorphisms. Since we only
are interested in limits with $v(G)\to\infty$, it is
equivalent to consider injective maps only. By inclusion-exclusion, it
is further equivalent to consider $\tind(F,G)$, defined as
the probability that a random injective map $V(F)\to V(G)$
maps $F$ isomorphically onto an induced copy of $F$ in
$G$; in other words, $\tind(F,G)$ equals the number of
labeled induced copies of $F$ in $G$ divided by 
the falling factorial $v(G)\dotsm(v(G)-v(F)+1)$.
Then
$\tind(F,\cdot)$ extends by continuity to $\cuq$, and
by inclusion-exclusion, 
for graph limits $\gG\in\cuoo$, 
$\tind(F,\gG)$ can be written as a linear combination of
$t(F_i,\gG)$ for subgraphs $F_i\subseteq F$.
We can define $\tind$ for bipartite graphs in the same way;
further details are in \cite{BCL1} and \cite{DJ}.

\begin{theorem}\label{TaltC4}
  \begin{thmenumerate}
\item
Let $\gG\in\cuoo$; \ie, $\gG$ is a graph limit. Then
$\gG\in\ctoo$ if and only if
$\tind(P_4,\gG)=\tind(C_4,\gG)=\tind(2K_2,\gG)=0$.
\item
Let $\gG\in\cboo$; \ie, $\gG$ is a
bipartite graph limit. Then
$\gG\in\ctboo$ if and only if
$\tind(2K_2,\gG)=0$.
  \end{thmenumerate}
\end{theorem}

In view of \ref{D1induced} and \ref{B1induced}, this is a special case of the
following simple general statement. 

\begin{theorem}\label{TUF}
Let $\cF=\set{F_1,F_2,\dots}$ be a finite or
infinite family of graphs, and let
$\cU_\cF\subseteq\cU$ be the set of all graphs that
do not contain any graph from $\cF$ as an induced subgraph,
\ie,
\begin{equation*}
\cU_\cF\=\set{G\in\cU:\tind(F,G)=0\text{ for } F\in\cF}.  
\end{equation*}
Let $\cuq_\cF$ be the closure of $\cU_\cF$ in $\cuq$.
Then
\begin{equation*}
\cuq_\cF\=\set{\gG\in\cuq:\tind(F,\gG)=0 \text{ for } F\in\cF}.  
\end{equation*}
In other words, if $\gG\in\cuoo$ is a graph limit, then $\gG$ is a limit of
a sequence of graphs in $\cU_\cF$ if and only if 
$\tind(F,\gG)=0$ for $F\in\cF$.

Conversely, if $\gG\in\cuq_\cF\cap\cuoo$ is
represented by a 
function
 $W$, then the random graph $\gwn\in\cU_\cF$
(almost surely).

The same results hold in the bipartite case.
\end{theorem}

\begin{proof}
  If $G_n\to\gG$ with $G\in\cU_\cF$, then
  $t(F,\gG)=\lim_\ntoo t(F,G_n)=0$ for every
  $F\in\cF$, by the continuity of $t(F,\cdot)$.

Conversely, suppose that $\gG\in\cuoo$ and
$t(F,\gG)=0$ for $F\in\cF$, and let $\gG$ be represented by a
function
$W$. It follows from \eqref{tuw} that if
$F\in\cF$ then
$\E t(F,\gwn)=t(F,\gG)=0$, and thus
$t(F,\gwn)=0$ a.s.; consequently
$\gwn\in\cU_\cF$ a.s. This proves the second
statement. Since $\gwn\to\gG$ a.s., it also shows that
$\gG$ is the limit of a sequence in $\cU_\cF$,
and thus $\gG\in\cuq_\cF$, which completes the proof of the first part.
\end{proof}

\section{Degree distributions}\label{Sdeg}
The results in this section hold for general graphs, they are applied to threshold
graphs in section \refS{Slim}. 

Let $\cP$ be the set of probability measures on $\oi$, equipped with
the standard topology of weak convergence, which makes $\cP$ a compact
metric space (see \eg{} Billingsley \cite{Bill}).

If $G$ is a graph, let $d(v)=d_G(v)$ denote the degree of vertex $v\in V(G)$,
and let $D_G$ denote the random variable defined as the degree
$d_G(v)$ of  a randomly chosen vertex $v$ (with the uniform
distribution on $V(G)$). Thus $0\le D_G\le v(G)-1$.
For a bipartite graph we similarly define $\dgg j$
as the degree $d_G(v)$ of  a randomly chosen vertex $v\in V_j(G)$,
$j=1,2$.
Note that $0\le \dgg1\le v_2(G)$ and $0\le \dgg2\le v_1(G)$.
Since we are interested in dense graphs, we will normalize these 
random degrees
to
$D_G/v(G)$ and, in the bipartite case, $\dgg1/v_2(G)$ and
$\dgg2/v_1(G)$; these are random variables in \oi.
The distribution of $D_G/v(G)$ will be called the (normalized) degree
distribution of $G$ and denoted by $\nu(G)\in\cP$; 
in other words, $\nu(G)$ is the empirical distribution function of
\set{d_G(v)/v(G): v\in V(G)}.
In the bipartite case we
similarly have two 
(normalized) degree distributions: $\nu_1(G)$ for $V_1(G)$ and
$\nu_2(G)$ for $V_2(G)$. 

The moments of the degree distribution(s) are given by the functional
$t(F,\cdot)$ for stars $F$, as stated in the following lemma.
We omit the proof, which is a straightforward consequence of the definitions.
\begin{lemma}\label{Ldegree}
  The moments of $\nu(G)$ are given by
  \begin{equation}\label{degmom}
	\intoi t^k\dd\nu(G)(t)
=t(K_{1,k},G),
\qquad k\ge1,
  \end{equation}
where $K_{1,k}$ is a star with $k$ edges.

In the bipartite case, similarly, for $k\ge1$,
  \begin{align}\label{degmomb}
& \intoi t^k\dd\nu_1(G)(t)
=t(K_{1,k},G),
&&
	\intoi t^k\dd\nu_2(G)(t)
=t(K_{k,1},G).
 \end{align}
\end{lemma}

This enables us to extend the definition of the (normalized) degree
distribution to 
the limit objects by continuity.

\begin{theorem}\label{Tdegree}
If $G_n$ are graphs with $v(G_n)\to\infty$ and $G_n\to \Gamma$ for some
$\Gamma\in\cuq$ as \ntoo, then 
$\nu(G_n)\to \nu(\Gamma)$ for some distribution $\nu(\Gamma)\in \cP$.  
This defines the `degree distribution' $\nu(\Gamma)$ (uniquely) for every graph
limit $\Gamma\in\cuoo$, and $\Gamma\mapsto\nu(\Gamma)$ is a continuous
map $\cuoo\to\cP$. Furthermore, \eqref{degmom} holds for all $G\in\cuq$.

Similarly, in the bipartite case, $\nu_1$ and $\nu_2$  extend
to continuous maps
$\cbq\to\cP$ such that \eqref{degmomb} holds for all $G\in\cbq$.
Furthermore, $\nu_2(\gG)=\nu_1(\gG\sss)$ for  $\gG\in\cbq$.
\end{theorem}

\begin{proof}
  An immediate consequence of \refL{Ldegree} and the method of
  moments. The last sentence follows from \eqref{degmomb} and \eqref{sss}.
\end{proof}

\begin{remark*}
  \refT{Tdegree} says that the degree distribution 
$\nu$ is a \emph{testable} graph parameter in the sense of \Borgsetal{}
  \cite{BCL1}, see in particular
  \cite[Section 6]{BCL1}.
(Except that $\nu$ takes values in $\cP$ instead of $\bbR$.)
\end{remark*}

If $\gG$ is represented by a function $W$ on $\oii$, we can easily find its
degree distribution from $W$.

\begin{theorem}\label{Tdw}
If\/ $W\in\cws$, then $\nu(\uw)$ equals the distribution of
$\intoi W(U,y)\dd y$, where $U\sim U(0,1)$.

Similarly, in the bipartite case,  
if\/ $W\in\cw$, then 
$\nu_1(\uww)$ equals the distribution of $\intoi W(U,y)\dd y$
and
$\nu_2(\uww)$ equals the distribution of $\intoi W(x,U)\dd x$.
\end{theorem}

\begin{proof}
  By \eqref{degmom} and \eqref{tuw},
  \begin{equation*}
	\begin{split}
\intoi t^k\dd\nu(\uw)(t)
&=t(K_{1,k},\uw)
=
\int_{\oi} \lrpar{\int_{\oi} W(x,y) \dd y}^k \dd x	  
\\&
=
\E \lrpar{\int_{\oi} W(U,y) \dd y}^k
	\end{split}
  \end{equation*}
for every $k\ge1$, and the result follows.
The bipartite case is similar, using \eqref{tuww}.
\end{proof}

If a graph $G$ has $n$ vertices, its number of edges is
\begin{equation*}
  \begin{split}
	|E(G)|=\half\sum_{v\in V(G)}d(v)
=\frac n2\E D_G
=\frac {n^2}2\E (D_G/n)
=\frac {n^2}2\intoi t\dd\nu(G)(t).
  \end{split}
\end{equation*}
Hence, the edge density of $G$ is
\begin{equation}\label{density}
  \begin{split}
	|E(G)|/\binom n2
=\frac n{n-1}\intoi t\dd\nu(G)(t).
  \end{split}
\end{equation}
If $(G_n)$ is a sequence of graphs with $v(G_n)\to\infty$ and
$G_n\to\gG\in\cuoo$, we see from \eqref{density} and \refT{Tdegree}
that the graph densities converge to
$\intoi t\dd\nu(\gG)(t)$, the mean of the distribution $\nu(\gG)$,
which thus may be called the (edge) density of $\gG\in\cuoo$.

If $\gG$ is represented by a function $W$ on $\oii$, \refT{Tdw} yields
the following.

\begin{corollary}
  $\gG_W$ has edge density $\iint_{\oii} W(x,y)\dd x\dd y$ for every $W\in\cws$.
\end{corollary}
\begin{proof}
  By \refT{Tdw}, the mean of $\mu(\gG_W)$ equals
  \begin{equation*}
\E \intoi W(U,y)\dd y	
=\intoi\intoi W(x,y)\dd x\dd y.
\qedhere
  \end{equation*}
\end{proof}

\section{Limits of threshold graphs}\label{Slim}

Recall from \refSS{SSthrlim} that $\ctoo$
is the set of  limits of threshold graphs,
and $\ctboo$ is the set 
of limits of \btg{s}.
Our purpose in this
section is to characterize the threshold graph limits, \ie{}
the elements of $\ctoo$ and $\ctboo$, and give
simple criteria for the convergence of a sequence of threshold graphs
to one of these limits.
We begin with some definitions.

A function $W:\oii\to\bbR$ is \emph{increasing} if $W(x,y)\le
W(x',y')$ whenever $0\le x\le x'\le 1$ and $0\le y\le y'\le1$.
A set $S\subseteq\oii$ is increasing if its indicator $\etta_S$ is an
increasing function on $\oii$, \ie, if $(x,y)\in S$ implies $(x',y')\in S$
whenever $0\le x\le x'\le 1$ and $0\le y\le y'\le1$.

If $\mu\in\cP$, let $\fmu$ be its distribution function
$\fmu(x)\=\mu([0,x])$, and let $\fmu(x-)\=\mu([0,x))$ be its
left-continuous version. Thus $\fmu(0-)=0\le\fmu(0)$ and $\fmu(1-)\le
1=\fmu(1)$. Further, let $\fmui:\oi\to\oi$ be the right-continuous
inverse defined by
\begin{equation}\label{fmui}
\fmui(x)\=\sup\set{t\le 1:\fmu(t)\le x}.   
\end{equation}
Note that $\fmui(0)\ge0$
and $\fmui(1)=1$.
Finally, define
\begin{equation}\label{smu1}
  \smu\=\bigset{(x,y)\in\oii:x\ge\fmu\bigpar{(1-y)-}}.
\end{equation}
It is easily seen that $\smu$ is a closed increasing subset of $\oii$
and that it contains the upper and right edges \set{(x,1)} and \set{(1,y)}.
Since $x\ge\fmu\bigpar{(1-y)-} \iff \fmui(x)\ge1-y$, we also have
\begin{equation}\label{smu2}
  \smu
=\bigset{(x,y)\in\oii:\fmui(x)+y\ge1}.
\end{equation}

We further write $\wmu\=\etta_{\smu}$ and let $\ggmub\=\ggb(\wmu)$
and, when $W$ is symmetric, 
$\ggmu\=\gG(\wmu)$. We denote the interior of a set $S$ by $S\inre$.
It is easily verified from \eqref{smu1} that
\begin{equation}\label{smuo}
  \smu\inre
=\bigset{(x,y)\in(0,1)^2:x>\fmu(1-y)}.
\end{equation}

Recall that the \emph{Hausdorff distance} between two non-empty compact
subsets $K_1$ and $K_2$ of some metric space $\cS$ 
is defined by 
\begin{equation}\label{hausdorff}
  \begin{split}
d_H(K_1,K_2)
&\=
\max\bigpar{\max_{x\in K_1}d(x,K_2),\,\max_{y\in K_2} d(y,K_1)}. 
  \end{split}
\end{equation}
This defines a metric on the set of all non-empty compact subsets of $\cS$. If
$\cS$ is compact, the resulting topology on the set of compact subsets
of $\cS$ (with the empty set as an isolated point) is compact and
equals the \emph{Fell topology} 
(see \eg{} \cite[Appendix A.2]{Kallenberg})
on the set of all closed subsets of $\cS$.

Let $\gl_d$ denote the Lebesgue measure in $\bbR^d$.
For measurable subsets $S_1,S_2$ of $\oii$, we also consider their 
\emph{measure distance} $\gl_2(S_1\gD S_2)$.
This equals the $L^1$-distance of their indicator functions, and is
thus a metric modulo null sets. 

For functions in $\cW$ we also use two different metrics: the
\emph{$L^1$-distance} $\int_{\oii}|W_1(x,y)-W_2(x,y)|\dd x\dd y$ 
and, in the symmetric case, the \emph{cut-distance} $\gdc$ defined by
\Borgsetal{} \cite{BCL1}, and in the bipartite case its analogue
$\gdcb$, see \refS{Sprel}.
Note that the cut-distance is only a pseudo-metric, since the distance
of two different functions may be 0. Note further that the
cut-distance is less than or equal to the $L^1$-distance.

We can now prove one of our main results, giving several  related 
characterizations of threshold graph limits. There are two versions, 
since we treat the bipartite case in parallel. 

\subsection*{The bipartite case}
It is convenient to begin with the bipartite case. 

\begin{theorem}
  \label{Tboo}
There are
bijections between the set $\ctboo$ of graph limits of
\btg{s} and each of the following sets. 
\begin{romenumerate}
  \item\label{TBoop}
The set $\cP$ of probability distributions on $\oi$.
\item\label{TBooc}
The set $\cbc$ of increasing closed sets $S\subseteq\oii$ that contain the
upper and right edges $\oi\times\set1\cup\set1\times\oi$.
\item\label{TBooo}
The set $\cbo$ of increasing open sets $S\subseteq(0,1)^2$.
\item\label{TBoow}
The set $\cbw$ of increasing \oix{} valued functions $W:\oii\to\set{0,1}$
modulo \aex{} equality.
\end{romenumerate}
More precisely, there are commuting bijections between these sets
given by the following mappings and their compositions:
\begin{equation}\label{mapb}
  \begin{aligned}
\ixbp&:\ctboo\to\cP,
&
\ixbp(\gG)&\=\nu_1(\gG)	;
\\
\ixpc&:\cP\to\cbc,
& \ixpc(\mu)&\=\smu; 
\\
\ixco&:\cbc\to\cbo,
&\ixco(S)&\=S\inre;
\\
\ixcw&:\cbc\to\cbw,
&\ixcw(S)&\=\etta_S;
\\
\ixow&:\cbo\to\cbw,
&\ixow(S)&\=\etta_S;
\\
\ixwb&:\cbw\to\ctboo,
&\ixwb(W)&\=\ggb_W.
  \end{aligned}
\end{equation}

\begin{diagram}
\ctboo &\rTo^{\ixbp}& \cP &\rTo^{\ixpc}& \cbc \\
 &  \luTo^{\ixwb} & & \ldTo_{\ixcw} & \dTo_{\ixco}\\
  & & \cbw& \lTo_{\ixow} &\cbo\\
\end{diagram}
In particular,
a probability distribution $\mu\in\cP$
corresponds to $\ggmub\in\ctboo$ and to
$\smu\in\cbc$,
$\smu\inre\in\cbo$, and
$\wmu\in\cbw$.
Conversely, $\gG\in\ctboo$ corresponds to $\nu_1(\gG)\in\cP$.
Thus, the mappings $\Gamma\mapsto\nu_1(\gG)$ and $\mu\mapsto\ggmub$
are the inverses of each other.

Moreover, these bijections are homeomorphisms,
with any of the following topologies or metrics:
the standard (weak) topology on $\cP$; the 
Hausdorff metric, 
or the Fell topology,
or the measure distance
on $\cbc$; 
the measure distance
on $\cbo$; 
the $L^1$-distance or the cut-distance 
on the set $\cbw$. 
\end{theorem}

\begin{proof}

The mappings in \eqref{mapb} are all well-defined, except
that we do not yet know that $\ixwb$ maps $\cbw$
into $\ctboo$.
We thus regard $\ixwb$ as a map $\cbw\to\cboo$
and let $\cbz\=\ixwb(\cbw)$ be its image; we 
will identify this as $\ctboo$ later.
For the time being we also regard $\ixbp$ as defined on
$\cbz$ (or on all of $\cboo$).

Consider first 
$\ixpc:\cP\to\cbc$.
By \eqref{smu1}, 
$\smu$ determines $\fmu$ at all continuity points,
and thus it determines $\mu$. Consequently, $\ixpc$
is injective.

If $S\in\cbc$ and $y\in\oi$, then \set{x:(x,y)\in S} is a closed subinterval
of $\oi$ that contains $1$, and thus $S=\set{(x,y)\in\oii:x\ge g(y)}$
for some function $g:\oi\to\oi$. Moreover, 
$g(1)=0$, $g$ is decreasing, \ie{}  $g(y_2)\le g(y_1)$ if $y_1\le y_2$, 
and, since $S$ is closed,
$g$ is right-continuous. Thus $g(1-x)$ is increasing and left-continuous, and
hence there exists a probability measure $\mu\in\cP$
such that $\fmu(x-)=g(1-x)$, $x\in\oi$. By \eqref{smu1}, then 
\begin{equation*}
\ixpc(\mu)=\smu=\set{(x,y)\in\oii:x\ge g(y)}=S.  
\end{equation*}
Hence $\ixpc$ is
onto. 
Consequently, $\ixpc$ is a bijection of $\cP$ onto $\cbc$.

If $S_1$ and $S_2$ are two different sets in $\cbc$, then there exists
a point $(x,y)\in S_1\setminus S_2$, say. There is a small open
disc with center in $(x,y)$ that does not intersect $S_2$, and since
$S_1$ is increasing, at least a quarter of the disc is contained in
$S_1\setminus S_2$. Hence, $\gl_2(S_1\gD S_2)>0$. 
Similarly, 
if $S_1$ and $S_2$ are two different sets in $\cbo$
and $(x,y)\in S_1\setminus S_2$,  then there is a small open
disc with center in $(x,y)$ that is contained in $S_1$, and since
$S_2$ is increasing, at least a quarter of the disc is contained in
$S_1\setminus S_2$, whence $\gl_2(S_1\gD S_2)>0$. 
This shows that the
measure distance is a metric on $\cbc$ and on $\cbo$, and that the
mappings $\ixcw$ and $\ixow$ into $\cbw$ are injective (remember that
\aex{} equal functions are identified in $\cbw$).

Next, let $S\subseteq\oii$ be increasing.  If $(x,y)\in\bS$ with
$x<1$ and $y<1$, it is easily seen that
$(x,x+\gd)\times(y,y+\gd)\subseteq S$ for $\gd=\min\set{1-x,1-y}$, and
thus
$(x,x+\gd)\times(y,y+\gd)\subseteq S\inre$.
It follows that, for any real $a$, the intersection of the boundary
$\partial S\=\bS\setminus S\inre$ with the diagonal  line
$L_a\=\set{(x,x+a):x\in\bbR}$ consists of at most two points (of which
one is on the boundary of $\oii$). In particular, 
$\gl_1(\partial S\cap L_a)=0$ and thus
\begin{equation}\label{ds}
 \gl_2(\partial S)
=
2\qqw \int_{-1}^1\gl_1(\partial S\cap L_a)\dd a
=0.
\end{equation}
Consequently, $\partial S$ is a null set for every increasing $S$.
Among other things, this shows that if $S\in\cbc$, then
$\ixow\ixco(S)=\etta_{S\inre}=\etta_S$ a.e.
Since elements of $\cbw$ are defined modulo \aex{} equality, this
shows that $\ixow\ixco=\ixcw:\cbc\to\cbw$.

If $W\in\cbw$, and thus $W=\etta_S$ for some
increasing $S\subseteq\oii$, let 
\begin{equation}
  \label{ts}
\tS\=\overline{S\cup \oi\times\set1\cup\set1\times\oi}.
\end{equation}
Then $\tS\in\cbc$ and \eqref{ds} implies that 
$\ixcw(\tS)=\etta_{\tS}=\etta_S=W$ a.e.
Similarly, $S\inre\in\cbo$ and
$\ixow(S\inre)=\etta_S=W$ a.e.
Consequently, $\ixcw$ and $\ixow$ are onto, and thus bijections.
Similarly (or as a consequence), $\ixco$ is a bijection of $\cbc$ onto
$\cbo$, with inverse $S\mapsto\tS$ given by \eqref{ts}.

Note that the composition $\ixcw\ixpc$ maps
$\mu\mapsto\etta_{\smu}=\wmu$, and 
let $\ixpb$ be the composition $\ixwb\ixcw\ixpc:\mu\mapsto\ggb(\wmu)=\ggmub$
mapping $\cP$ into $\cboo$. 
Since $\ixpc$ and $\ixcw$ are bijections, its image
$\ixbp(\cP)=\ixwb(\cbw)=\cbz\subseteq\cboo$. 

If $\mu\in\cP$, then the composition
$\ixbp\ixpb(\mu)=\nu_1(\ggmub)$ equals by 
\refT{Tdw} and \eqref{smu2} the distribution of 
\begin{equation}
  \intoi\etta_{\smu}(U,y)\dd y = \fmui(U).
\end{equation}
As is well-known, and easy to verify using \eqref{fmui}, this
distribution equals $\mu$. 
Hence, the composition $\ixbp\ixpb$ is the identity.
It follows that $\ixpb$ is injective and thus a
bijection of $\cP$ onto its image $\cbz$, and that $\ixbp$ (restricted to
$\cbz$)
is its inverse.

We have shown that all mappings in \eqref{mapb} are bijections, except
that we have not yet shown that $\cbz=\ctboo$. We next show that the mappings
are homeomorphisms.

Recall that the topology on $\cP$ can be defined by the \Levy{} metric
defined by (see \eg{} \cite[Problem 5.25]{Gut})
\begin{equation}\label{levy}
  d_L(\mu_1,\mu_2)\=\inf\set{\eps>0:
F_{\mu_1}(x-\eps)-\eps \le F_{\mu_2}(x)\le F_{\mu_1}(x+\eps)+\eps
\text{ for all $x$}}.
\end{equation}
If $\mu_1,\mu_2\in\cP$ with $d_L(\mu_1,\mu_2)<\eps$, 
it follows from \eqref{smu1} and \eqref{levy}
that if $(x,y)\in S_{\mu_1}$ and $x,y<1-\eps$,
then
\begin{equation*}
  F_{\mu_2}\bigpar{(1-y-\eps)-}
\le
  F_{\mu_1}\bigpar{(1-y)-}+\eps
\le x+\eps
\end{equation*}
and thus $(x+\eps,y+\eps)\in S_{\mu_2}$. Considering also the simple cases 
$x\in[1-\eps,1]$ and $y\in[1-\eps,1]$, it follows that if
$(x,y)\in S_{\mu_1}$, then $d\bigpar{(x,y),S_{\mu_2}}\le\sqrt2\,\eps$.
Consequently, by \eqref{hausdorff} and symmetry,
\begin{equation*}
  d_H(S_{\mu_1},S_{\mu_2})\le\sqrt2\, d_L(\mu_1,\mu_2),
\end{equation*}
which shows that $\ixpc$ is continuous if $\cbc$ is given the topology
given by the Hausdorff metric.

The same argument shows that for any $(x_0,y_0)$, the intersection of
the difference $S_{\mu_1}\gD S_{\mu_2}$ with the diagonal line $L_a$
defined above is an interval of length at most 
$\sqrt2\,d_L(\mu_1,\mu_2)$, and thus, by integration over $a$ as in
\eqref{ds}, 
\begin{equation*}
 \gl_2\bigpar{S_{\mu_1}\gD S_{\mu_2}}\le 2d_L(\mu_1,\mu_2). 
\end{equation*}
Hence, $\ixpc$ is continuous also if $\cbc$ is given the topology
given by the measure distance.

Since $\cP$ is compact and $\ixpc$ is a bijection, it follows that
$\ixpc$ is a homeomorphism for both these topologies on $\cbc$. In
particular, these topologies coincide on $\cbc$. As remarked before
the theorem, since $\oii$ is compact, also the Fell topology coincide
with these on $\cbc$.

The bijections $\ixco$, $\ixcw$ and $\ixow$ are isometries for the
measure distance on $\cbc$ and $\cbo$ and the $L^1$-distance on
$\cbw$, and thus homeomorphisms. Furthermore, 
still using 
the $L^1$-distance on $\cbw$,
it is easily seen from 
\eqref{tuww}, as for the symmetric case in \cite{LSz}, \cite{BCL1},
that for every fixed bipartite graph $F$, the mapping
$W\mapsto t(F,\ggb_W)$ is continuous, which by definition of the
topology in $\cboo$ means that $\ixwb:W\mapsto\ggb_W$ is continuous.
Hence, the bijection $\ixwb$ is a homeomorphism of the compact space
$\cbw$ onto its image $\cbz$.

As said above,
the cut-distance is only a pseudo-metric on $\cW$. But two functions
in $\cW$ with cut-distance 0 are mapped onto the same element in
$\cboo$, and since we have shown that $\ixwb$ is injective on $\cbw$,
it follows that the restriction of the cut-distance to $\cbw$ is a
metric. Moreover, the identity map on $\cbw$ is continuous from the
$L^1$-metric to the cut-metric, and since the space is compact under
the former metric, the two metrics are equivalent on $\cbw$.

We have shown that all mappings are homeomorphisms. It remains only to
show that $\cbz=\ctboo$.
To do this, observe first that if $G$ is a \btg, and we
order its vertices in each of the two vertex sets with increasing
vertex degrees, 
then the function $W_G$ defined in \refS{Sprel} is increasing and
belongs thus to $\cbw$. Consequently, $\pib(G)=\ixwb(W_G)\in\cbz$.
If $\gG\in\ctboo$, then by definition there exists a sequence $G_n$ of
\btg{s} with $v_1(G_n),v_2(G_n)\to\infty$ such that $G_n\to\gG$ in
$\cbq$. This implies that $\pib(G_n)\to\gG$ in $\cboo$, and since
$\pib(G_n)\in\cbz$ and $\cbz$ is compact and thus a closed subset of
$\cboo$, we find $\gG\in\cbz$.

Conversely, if $\gG\in\cbz$, then $\gG=\ixwb\ixcw(S)$ for some set
$S\in\cbc$. For each $n$, partition $\oii$ into $n^2$ closed squares $Q_{ij}$
of side $1/n$, and let $S_n$ be the union of all $Q_{ij}$ that
intersect $S$. Then $S_n\in\cbc$, $S\subseteq S_n$ and $d_H(S_n,S)\le
\sqrt2/n$. Let $W_n\=\etta_{S_n}=\ixcw(S_n)$ and let
$\gG_n\=\ixwb(W_n)\in\cbz$. Since $\ixcw$ and $\ixwb$ are 
continuous, $W_n\to W\=\etta_S$ in $\cbw$ and $\gG_n\to\ixwb(W)=\gG$ in
$\cbz\subset\cboo$. However, $W_n$ is a step function of the form
$W(G_n)$ for some bipartite graph $G_n$ with $v_1(G_n)=v_2(G_n)=n$,
and thus $\pib(G_n)=\ggb_{W_n}=\gG_n$. 
Moreover, each $S_n$ and thus each $W_n$ is increasing, and hence
$G_n$ is a \btg.
Since $\pib(G_n)=\gG_n\to\gG$ in
$\cboo$, it follows that $G_n\to\gG$ in $\cbq$, and thus $\gG\in\ctboo$.

Consequently, $\cbz=\ctboo$, which completes the proof.
\end{proof}

\begin{remark}
  Another unique representation by increasing closed sets is given by
  the family $\cbc'$ of closed increasing subsets $S$ of $\oii$ that
  satisfy $S=\overline{S\inre}$; there are bijections
  $\cbc'\to\cbo$ and $\cbo\to\cbc'$ given
  by $S\mapsto S\inre$ and $S\mapsto \overline S$. 
We can, again, use the measure
  distance on $\cbc'$, but not the Hausdorff distance.
(For example, $[0,1]\times[1-\eps,1]\to\emptyset$ in
  $\cbc'$ as $\eps\to0$.)
\end{remark}

\begin{corollary}
  The degree distribution yields a homeomorphism $\Gamma\mapsto\nu_1(\Gamma)$
of $\ctboo$ onto $\cP$.
\end{corollary}

Of course, 
$\Gamma\mapsto\nu_2(\Gamma)=\nu_1(\Gamma\sss)$
yields another homeomorphism of $\ctboo$ onto $\cP$.
To see the connection between these, and (more importantly) to prepare for the
corresponding result in the non-bipartite case,
we investigate further the reflection involution.

If $S\subseteq\oii$, let $S\sss:=\set{(x,y):(y,x)\in S}$ 
be the set $S$ reflected in the main diagonal. Thus
$\etta_{S\sss}=\etta_S\sss$. We have defined the reflection map $\sssx$
for bipartite graphs and graph limits, and for the sets and functions 
in \refT{Tboo}\ref{TBooc}\ref{TBooo}\ref{TBoow}, and it is easily seen
that these correspond to each other by the bijections in \refT{Tboo}.
Consequently, there is a corresponding map (involution)
$\mu\mapsto\mu\sss$ of 
$\cP$ onto itself too. This map is less intuitive than the others; to
find it explicitly, we find from \eqref{smu1}, \eqref{smu2} and
$S_{\mu\sss}=\smu\sss$ that
\begin{equation*}
 x\ge \fmusss\bigpar{(1-y)-} \iff
(y,x)\in\smu\iff
\fmui(y)+x\ge1 
\end{equation*}
and thus
$\fmusss\bigpar{(1-y)-}=1-\fmui(y)$
and
\begin{equation}\label{musss}
 \fmusss(t)=1-\fmui\bigpar{(1-t)-},
\qquad
0\le t\le1.
\end{equation}
This means that the graph of the  distribution function is reflected
about the diagonal between $(0,1)$ and $(1,0)$ (and adjusted at the jumps).

The map $\sssx$ is continuous on $\cP$, by \refT{Tboo} and the obvious
fact that $S\mapsto S\sss$ is continuous on, for example, $\cbc$.

We let $\cps\=\set{\mu\in\cP:\mu\sss=\mu}=\set{\mu\in\cP:\smu=\smu\sss}$ 
be the set of probability
distributions invariant under the involution $\sssx$.
Since $\sssx$ is continuous, $\cps$ is a closed and thus compact subset
of $\cP$.

\begin{remark}
If $\mu\in\cps$, let
$x_0\=1-\inf\set{x:(x,x)\in\smu}$.
Then \eqref{smu1} and \eqref{smuo} imply that
$\fmu(x_0-)\le1-x_0\le\fmu(x_0)$, and the restriction of $\fmu$ to
$[0,x_0)$ is an increasing right-continuous
  function with values in $[0,1-x_0]$ and this restriction
  determines $\fmu(t)$ for $x\ge x_0$ too by
  \eqref{musss}.

Conversely, given any $x_0\in\oi$ and increasing
right-continuous $F:[0,x_0)\to[0,1-x_0]$, there is a unique
  $\mu\in\cps$ with $\fmu(x)=F(x)$ for
  $x<x_0$ and $\fmu(x_0)\ge 1-x_0$.
  \end{remark}

\subsection*{Non-bipartite case}
We can now state our main theorem for (non-bipartite) threshold graph limits.

\begin{theorem}
  \label{Too}
There are bijections between the set $\ctoo$ of graph limits of
\tg{s} and each of the following sets. 
\begin{romenumerate}
  \item\label{Toop}
The set $\cps$ of probability distributions on $\oi$ symmetric with
respect to $\sssx$.
\item\label{Tooc}
The set $\ctc$ of symmetric increasing closed sets $S\subseteq\oii$
that contain the 
upper and right edges $\oi\times\set1\cup\set1\times\oi$.
\item\label{Tooo}
The set $\cto$ of symmetric increasing open sets $S\subseteq(0,1)^2$.
\item\label{Toow}
The set $\ctw$ of symmetric increasing \oix{} valued functions
$W:\oii\to\set{0,1}$ modulo \aex{} equality.
\end{romenumerate}
More precisely, there are commuting bijections between these sets
given by the following mappings and their compositions:
\begin{equation}\label{map}
  \begin{aligned}
\ixtp&:\ctoo\to\cps,
&
\ixtp(\gG)&\=\nu(\gG)	;
\\
\ixpc&:\cps\to\ctc,
& \ixpc(\mu)&\=\smu; 
\\
\ixco&:\ctc\to\cto,
&\ixco(S)&\=S\inre;
\\
\ixcw&:\ctc\to\ctw,
&\ixcw(S)&\=\etta_S;
\\
\ixow&:\cto\to\ctw,
&\ixow(S)&\=\etta_S;
\\
\ixwt&:\ctw\to\ctoo,
&\ixwt(W)&\=\gG_W.
  \end{aligned}
\end{equation}
In particular,
a probability distribution $\mu\in\cps$
corresponds to $\ggmu\in\ctoo$ and to
$\smu\in\ctc$,
$\smu\inre\in\cto$, and
$\wmu\in\ctw$.
Conversely, $\gG\in\ctoo$ corresponds to $\nu(\gG)\in\cps$.
Thus, the mappings $\Gamma\mapsto\nu(\gG)$ and $\mu\mapsto\ggmu$
are the inverses of each other.

\begin{diagram}
\ctoo &\rTo^{\ixtp}& \cps &\rTo^{\ixpc}& \ctc \\
 &  \luTo^{\ixwt} & & \ldTo_{\ixcw} & \dTo_{\ixco}\\
  & & \ctw& \lTo_{\ixow} &\cto\\
\end{diagram}

Moreover, these bijections are homeomorphisms,
with any of the following topologies or metrics:
the standard (weak) topology on $\cps\subset\cP$; the 
Hausdorff metric, 
or the Fell topology,
or the measure distance
on $\ctc$; 
the measure distance
on $\cto$; 
the $L^1$-distance or the cut-distance 
on the set $\ctw$.
These homeomorphic topological spaces are compact metric spaces.
\end{theorem}

\begin{proof}
  The mappings $\ixpc$, $\ixco$, $\ixcw$, $\ixow$ are restrictions of
  the corresponding mappings in \refT{Tboo}, and it follows from
  \refT{Tboo} and the definitions that these mappings are bijections
  and homeomorphisms for the given topologies. The spaces are closed
  subspaces of the corresponding spaces in \refT{Tboo}, since $\sssx$
  is continuous on these spaces, and thus compact metric spaces.

The rest is as in the proof of \refT{Tboo}, and we omit some details.
It follows from \refT{Tdw} that
the composition $\ixwb\ixcw\ixpc:\mu\mapsto\gG(\wmu)=\ggmu$
is a bijection of $\cps$ onto a subset $\ctz$ of $\ctboo$,
with $\ixtp$ as its inverse. It follows that these mappings too are
homeomorphisms, and that the $L^1$-distance and cut-distance are
equivalent on $\ctw$.

To see that $\ctz=\ctoo$, we also follow the proof of \refT{Tboo}. A
minor complication is that if $G\in\cT$ is a threshold graph, and we
order the vertices with increasing degrees, then $W_G$ is not
increasing, because $W_G(x,x)=0$ for all $x$ since we consider
loopless graphs only. However, we can define $\wgx$ by changing $W_G$
to be 1 on some squares on the diagonal so that $\wgx$ is increasing
and thus $\wgx\in\ctw$, and the error $\|W_G-\wgx\|_{L^1}\le 1/v(G)$.
If we define $\pix(G)\=\gG(\wgx)\in\ctz$, we see that 
if $(G_n)$ is a sequence of threshold graphs with $v(G_n)\to\infty$, then
for every graph $F$, by a simple estimate, see \eg{} \cite[Lemma 4.1]{LSz},
\begin{equation}
| t(F,\pix(G_n))-t(F,G_n) |
\le e(F) \|W(G_n)-\wgnx\|_{L^1}\le e(F)/v(G_n)\to0.
\end{equation}
It follows that $G_n\to\gG$ in $\cuq$ if and only if $\pix(G_n)\to\gG$
in $\cuoo$. 
If $\gG\in\ctoo$, then there exists such a sequence $G_n\to \gG$, 
and thus $\pix(G_n)\to\gG$ in $\cuoo$, and since
$\pix(G_n)\in\ctz$ and $\ctz$ is compact, we find $\gG\in\ctz$.

The converse follows in the same way. If 
$\gG\in\ctz$, then $\gG=\ixwt(W)$ for some function
$W\in\ctw$. The approximating step functions $W_n$ constructed in the
proof of \refT{Tboo} are symmetric, and if we let $W_n\yy$ by the
modification that vanishes on all diagonal squares, $W_n\yy=W_{G_n}$
for some \tg{} $G_n$, and for every graph $F$,
\begin{equation*}
  t(F,G_n)
=t(F,W_n\yy)
=t(F,W_n)+o(1)
=t(F,W)+o(1).
\end{equation*}
Hence, $G_n\to\gG_W=\gG$ in $\cuq$, and thus $\gG\in\ctoo$.
Consequently, $\ctz=\ctoo$.
\end{proof}

\begin{corollary}
  The degree distribution yields a homeomorphism $\Gamma\mapsto\nu(\Gamma)$
of $\ctoo$ onto the closed subspace $\cps$ of $\cP$.
\end{corollary}

\begin{remark}
The fact that a graph limit $\gG$ can be represented by a
function
$W\in\ctw$ if and only if
$\tind(P_4,\gG)=\tind(C_4,\gG)=\tind(2K_2,\gG)=0$,
which by \refT{TaltC4} is equivalent to the bijection
$\ctoo\leftrightarrow \ctw$ in \refT{Too}, is
also proved by 
\Lovasz{} and Szegedy \cite{LSz:forcible}.  
\end{remark}

We have described the possible limits of sequences of \tg{s};
this makes it easy to see when such sequences converge.

\begin{theorem}  \label{Tlim}
  Let $G_n$ be a sequence of \tg{s} such that
  $v(G_n)\to\infty$. 
Then $G_n$ converges in $\cuq$ as \ntoo, if and only if 
the degree distributions $\nu(G_n)$ converge to some distribution
  $\mu$. In this case, $\mu\in\cps$ and $G_n\to\Gamma_\mu$.
\end{theorem}
 
\begin{proof}
  As in the proof of \refT{Too}, $G_n\to\gG$ if and only if
  $\pix(G_n)\to\gG$ in $\ctz=\ctoo$, which by \refT{Too} holds if and
  only if $\nu(\pix(G_n))\to\nu(\gG)$. 
By \refT{Tdw}, $\nu(\pix(G_n))$ equals the distribution of 
$\intoi \wgnx(U,y)\dd y$, but this random variable differs 
by at most $1/v(G_n)=o(1)$
from the random variable
  $\intoi W_{G_n}(U,y)\dd y$, which has degree distribution
$\nu(G_n)$. The result follows.
\end{proof}

\begin{theorem}  \label{TBlim}
  Let $G_n$ be a sequence of \btg{s} such that
  $v_1(G_n),v_2(G_n)\to\infty$. 
Then $G_n$ converges in $\cbq$ as \ntoo, if and only if 
the degree distributions $\nu_1(G_n)$ converge to some distribution
  $\mu$. In this case, 
$\nu_2(G_n)\to\mu\sss$ and
$G_n\to\ggmub$.
\end{theorem}
 
\begin{proof}
$G_n\to\gG$ if and only if
  $\pi(G_n)\to\gG$ in $\cbz=\cboo$, which by \refT{Tboo} holds if and
  only if $\nu_1(\pi(G_n))\to\nu_1(\gG)$. 
It follows from \refT{Tdw} that $\nu_1(\pi(G_n))=\nu_1(G_n)$, and the
  result follows from \refT{Tboo}.
\end{proof}

\begin{remark}\label{Rfinite}
  A \tg{} limit $\gG$ is, by 
\refT{Too}, determined by its degree distribution and
  the fact that it is a \tg{} limit. 
By \refT{TUF} and  \refL{Ldegree}, $\gG$ is
thus determined by
$t(F,\gG)$ for $F$ in the set
\set{P_4,C_4,2K_2,K_{1,1},K_{1,2},\dots}.
\Lovasz{} and Szegedy \cite{LSz:forcible} have shown that
in some special cases, a finite set of $F$ is enough; for
example, the limit defined by the function $W(x,y)=\ett{x+y\ge1}$
(see \refE{Euniform} and Figure \ref{fig:sps2})
is the unique graph limit with 
$t(P_4,\gG)=t(C_4,\gG)=t(2K_2,\gG)=0$,
$t(K_2,\gG)=1/2$,
$t(P_3,\gG)=1/3$.
\end{remark}

\section{Random threshold graphs}\label{Srandom} 

We consider several ways to define random threshold graphs.
We will only consider constructions with a fixed number $n$ of
vertices; in fact, we take the vertex set to be $[n]=\set{1,\dots,n}$,
where $n\ge1$ is a given parameter.
By a \emph{random \tg} we thus mean a random element of
$\ct_n\=\set{G\in\ct:V(G)=[n]}$ for some $n$; we do not imply any
particular construction or distribution unless otherwise stated. 
(We can regard these graphs as either labeled or unlabeled.)

This section treats four classes of examples: a canonical example based
on increasing sets, random weights examples, random attachment examples
and uniform random threshold graphs.

\subsection{Increasing set}\label{SSset}
  For any symmetric increasing $S\subseteq\oii$, we let $W=\etta_S$ and
  define $\tsn\=\gwn$ as in \refS{Sprel}.
In other words, we take \iid{} random variables $U_1,\dots,U_n\sim U(0,1)$ 
and  draw an edge $ij$ if $(U_i,U_j)\in S$.

As said in \refS{Sprel}, $\gwn\asto\gG_W$, which in this case means
  that
$\tsn\asto \gG(\etta_S)\in\ctoo$.
We denote $\gG(\etta_S)$ by $\gG_S$ and have thus
  the following result, using also \refT{Tdw}.

  \begin{theorem}\label{Ttsn}
	As \ntoo, $\tsn\asto\gG_S$. In particular,
the degree distribution
$\nu(\tsn)\asto\nu(\gG_S)$, which equals the
distribution of 
\begin{equation}\label{ttsn}
\gf_S(U)\=|\set{y:(U,y)\in S}|
=\P\bigpar{(U,U')\in S\mid U},
\end{equation}
with $U,U'\sim U(0,1)$ independent.
\nopf
  \end{theorem}

By \refT{Too}, this construction
gives a canonical representation of the limit
objects in $\ctoo$, and we may restrict ourselves to closed or open
sets as in \refT{Too}\ref{Tooc}\ref{Tooo} to get a unique
representation. We can obtain any desired degree distribution
$\mu\in\cps$ for the limit by choosing $S=\smu$.
This construction further gives a canonical representation of random
threshold graphs 
for finite $n$, provided we make two natural additional assumptions.

\begin{theorem}\label{Trand}
  Suppose that $(G_n)_1^\infty$ is a sequence of random threshold graphs with
  $V(G_n)=[n]$ such that the distribution of each $G_n$ is invariant
  under permutations of $[n]$ and that the restriction 
(induced subgraph)
of $G_{n+1}$ to  $[n]$ has the same distribution as $G_{n}$, for every $n\ge1$.
If further $\nu(G_n)\pto\mu$ as \ntoo, for some $\mu\in\cP$, then
  $\mu\in\cps$ and, for every
  $n$, $G_n\eqd\tsxn{\smu}$.
\end{theorem}

\begin{proof}
It follows from \refT{Tlim} that $G_n\pto \ggmu$. (To apply
\refT{Tlim} to convergence in probability, we can 
use the standard
trick of considering subsequences that converge \aex,
since every subsequence has such a subsubsequence
\cite[Lemma 4.2]{Kallenberg}.)

  If we represent a graph by its edge indicators, the random graph
  $G_n$ can be regarded as a family of \oix-valued random variables
indexed by pairs $(i,j)$, $1\le i<j\le n$. By assumption, these
  families for different $n$ are consistent, so by the Kolmogorov
extension
theorem \cite[Theorem 6.16]{Kallenberg}, they can be defined for all
  $n$ together, which means that there exists a random infinite graph
  $G_\infty$ with vertex set $\bbN$ whose restriction to $[n]$
  coincides (in distribution) with $G_n$. Moreover, since each $G_n$
is invariant under permutations of the vertices, so is $G_\infty$,
  \ie, $G_\infty$ is \exch.
By Aldous and Hoover \cite{Aldous:1985}, 
see also \cite{Kallenberg:exch} and \cite{DJ},
every \exch{} random infinite graph can be obtained as a mixture of
  $\gwoo$; in other words, as $\gwoo$ for some random function
  $W\in\cws$. In this case, the subgraphs $G_n$ converge in probability
  to the corresponding random $\gG_W$, see Diaconis and Janson \cite{DJ}.
Since we have shown that $G_n$ converge to a deterministic graph limit
  $\ggmu$, 
we can take $W$ deterministic so
it follows that $G_\infty\eqd\gwoo$ for some $W\in\cws$;
  moreover, $\ggmu=\gG_W$, and thus we can by \refT{Too} choose
  $W=\wmu$.
(Recall that in general, $W$ is not unique.)
Consequently,
\begin{equation*}
G_n\eqd\gxwx{n}{\wmu}=\tsxn{\smu}.
\qedhere
\end{equation*}
\end{proof}

\subsection{Random weights}\label{SSweights}
Definition \ref{D1.1} suggests immediately the construction \ref{DRG1}:

Let $X_1,X_2,\dots,$ be \iid{} copies of a random variable $X$, 
let $\taux\in\bbR$,
and 
let $\txtn$ be the threshold graph with vertex set $[n]$ and edges $ij$
for all pairs $ij$ such that $X_i+X_j>\taux$.
(We can without loss of generality let $\taux=0$, by replacing $X$ by
$X-\taux/2$.)

Examples \ref{Ex2} and \ref{Euniform} are in this mode.

Let $F(x)\=\P(X\le x)$ be the distribution function of $X$, and let
$F\qw$ be its right-continuous inverse defined
by
\begin{equation}\label{e2fmui}
F\qw(u)\=\sup\set{x\in\bbR:F(x)\le u}.
\end{equation}
(Cf.\ \eqref{fmui}, where we consider distributions on $\oi$ only.)
Thus $-\infty<F\qw(u)<\infty$ if $0<u<1$, while $F\qw(1)=\infty$.
It is  well-known 
that the random variables $X_i$ can be constructed as $F\qw(U_i)$ with $U_i$
independent uniformly distributed random variables on $(0,1)$, which
leads to the following theorem, showing that this construction is
equivalent to the one in \refSS{SSset} for a suitable
set $S$.
Parts of this theorem were found earlier by Masuda, Konno and
co-authors \cite{KMRS,MMK}. 

\begin{theorem}\label{Tkonst}
Let $S$ be the symmetric increasing set 
  \begin{equation}\label{jeppe}
	S\=\set{(x,y)\in(0,1]^2:F\qw(x)+F\qw(y)>\taux}.
  \end{equation}
Then $\txtn\eqd\tsn$ for every $n$.

Furthermore, as \ntoo, the degree distribution
$\nu(\txtn)\asto\mu$ and thus
$\txtn\asto\gG_\mu$, where
$\mu\in\cps$ is the distribution of the random variable
$1-F(\taux-X)$, \ie{} 
\begin{equation}\label{e2fmu}
 \mu[0,s]=\P\bigpar{1-F(\taux-X)\le s},
\qquad s\in\oi.
\end{equation}
\end{theorem}
\begin{proof}
  Taking $X_i=F\qw(U_i)$, we see that 
  \begin{equation*}
\text{there is an edge }ij
\iff
  F\qw(U_i)+F\qw(U_j)>\taux\iff(U_i,U_j)\in S,	
  \end{equation*} 
which shows that
  $\txtn=\tsn$.

The remaining assertions now follow from \refT{Ttsn} together
with the calculation, with $U,U'\sim U(0,1)$ independent and
$X=F\qw(U)$, $X'=F\qw(U')$,
\begin{equation*}
  \begin{split}
\gf_S(U)
&=\P\bigpar{(U,U')\in S\mid U}	
=\P\bigpar{F\qw(U)+F\qw(U')>\taux\mid U}	
\\&
=\P\bigpar{X+X'>\taux\mid X}	
=\P\bigpar{X'>\taux-X\mid X}	
=1-F(\taux-X).
  \end{split}
\end{equation*}
\end{proof}

The set $S$ defined in \eqref{jeppe} is in general neither
open nor closed; the corresponding open set is
\begin{equation*}
  S\inre=\bigset{(x,y)\in(0,1)^2:F\qw(x-)+F\qw(y-)>\taux},
\end{equation*}
and the corresponding closed set $\smu$ in \refT{Too} can be found as
$\tilde{S\inre}$ from 
\eqref{ts}. If we assume  for simplicity  that the distribution of $X$
is continuous, then, as is easily verified, 
\begin{equation*}
  \smu=\bigset{(x,y)\in\oii:F\qw(x)+F\qw(y)\ge\taux},
\end{equation*}
where we define $F\qw(1)=\infty$ (and interpret
$\infty+(-\infty)=\infty$ in case $F\qw(0)=-\infty$).
We can use these sets  instead of $S$ in \eqref{jeppe}
since they differ by null sets only.

\subsection{Random addition of vertices}\label{SSaddition}
Preferential  attachment graphs are a rich topic of research in modern graph theory. 
See the monograph \cite{ChungLu}, along with the survey
\cite{Mitzenmacher}. The versions in this section are natural because
of  \ref{D1.2} and the construction
 \ref{DRG2}.

Let $\tnp$ be the random threshold graph with $n$
vertices obtained by adding vertices one by one
with the new vertices chosen as isolated or dominating at random,
independently of each other and with a given probability $p\in\oi$ of
being dominating. (Starting with a single vertex, there are thus 
$n-1$ vertex additions.)

The vertices are not equivalent (for example, note
that the edges $1i$, $i\neq1$, appear independently, but not the
edges $ni$, $i\neq n$), so  we also define the random threshold graph
$\tnpx$ obtained by a random permutation of the vertices in $\tnp$.
(When considering unlabeled graphs, there is no difference between
$\tnp$ and $\tnpx$.)

\begin{remark}
  We may, as stated in \ref{DRG2}, use different
  probabilities $p_i$ for different vertices. We leave it to the
  reader to explore this case, for example with $p_i=f(i/n)$ for
  some given continuous function $f:\oi\to\oi$.
\end{remark}

\begin{theorem}\label{Ttnp}
  The degree distribution $\nu(\tnp)$ 
converges 
\as{}
as \ntoo{} to a distribution $\mu_p$ that, for $0<p<1$, has constant density
$(1-p)/p$ on $(0,p)$ and $p/(1-p)$ on $(p,1)$; $\mu_0$ is a point mass
at $0$ and $\mu_1$ is a point mass at $1$. In particular, $\mu_{1/2}$
is the uniform distribution on $\oi$.

Consequently, $\tnp\asto \gG_{\mu_p}\in\ctoo$.
\end{theorem}

\begin{proof}
Let $Z_n(t)$ be the number of vertices in \set{1,\dots,\floor{nt}}
that are added as dominating. It follows from the law of large numbers
that $n\qw Z_n(t)\asto pt$, uniformly on $\oi$, and we assume this in
the sequel of the proof.

If vertex $k$ was added as isolated, it has degree $Z_n(1)-Z_n(k/n)$,
since its neighbours are the vertices that later are added as dominating.
Similarly, if vertex $k$ was added as dominating,
it has degree $k-1+Z_n(1)-Z_n(k/n)$.
Consequently, if $\mu_n$ is the (normalized) degree distribution of
$\tnp$, and $\phi$ is any continuous 
function on $\oi$,
then
\begin{multline*}
    \intoi\phi(t)\dd\mu_n(t)
=\frac1n\sum_{k=1}^n\phi(d(k)/n)
\\
\shoveleft{\quad=\frac1n\sum_{k=1}^n
\Bigl(
\phi\bigpar{n\qw Z_n(1)-n\qw Z_n(k/n)}
\ett{\gD Z_n(k/n)=0}}
\\
+
\phi\bigpar{n\qw Z_n(1)-n\qw Z_n(k/n)+(k-1)/n}
\ett{\gD Z_n(k/n)=1}
\Bigr).	
\end{multline*}
Since $n\qw Z_n(t)\to pt$ uniformly, and $\phi$ is uniformly
continuous,
it follows that, as \ntoo,
\begin{equation*}
  \begin{split}
    \intoi\phi(t)\dd\mu_n(t)
&=\frac1n\sum_{k=1}^n
\Bigl(
\phi\bigpar{p(1-k/n)}
\ett{\gD Z_n(k/n)=0}
\\&\hskip4em
+
\phi\bigpar{p(1-k/n)+k/n}
\ett{\gD Z_n(k/n)=1}
\Bigr)
+o(1)
\\
&=
\intoi
\phi\bigpar{p(1-t)}\dd\bigpar{n\qw\floor{nt}-n\qw Z_n(t)}
\\&\hskip4em
+
\intoi
\phi\bigpar{p(1-t)+t}\dd\bigpar{n\qw Z_n(t)}
+o(1)
.	
  \end{split}
\end{equation*}
Since the convergence $n\qw Z_n(t)\to pt$ implies (weak) convergence of
the corresponding measures, we finally obtain, as \ntoo,
\begin{equation*}
  \begin{split}
    \intoi\phi(t)\dd\mu_n(t)
&\to
\intoi
\phi\bigpar{p(1-t)}(1-p)\dd t
+
\intoi
\phi\bigpar{p(1-t)+t}p\dd t
\\
&=\frac{1-p}p
\int_0^p \phi(x)\dd x
+\frac{p}{1-p}
\int_p^1 \phi(x)\dd x
\\&
=\intoi\phi(x)\dd\mu_p(x)
,	
  \end{split}
\end{equation*}
with obvious modifications if $p=0$ or $p=1$. 
\end{proof}

\begin{figure}[ht] 
   \hskip-1cm
   \begin{minipage}{2.6in}
   \includegraphics[width=2.6in]{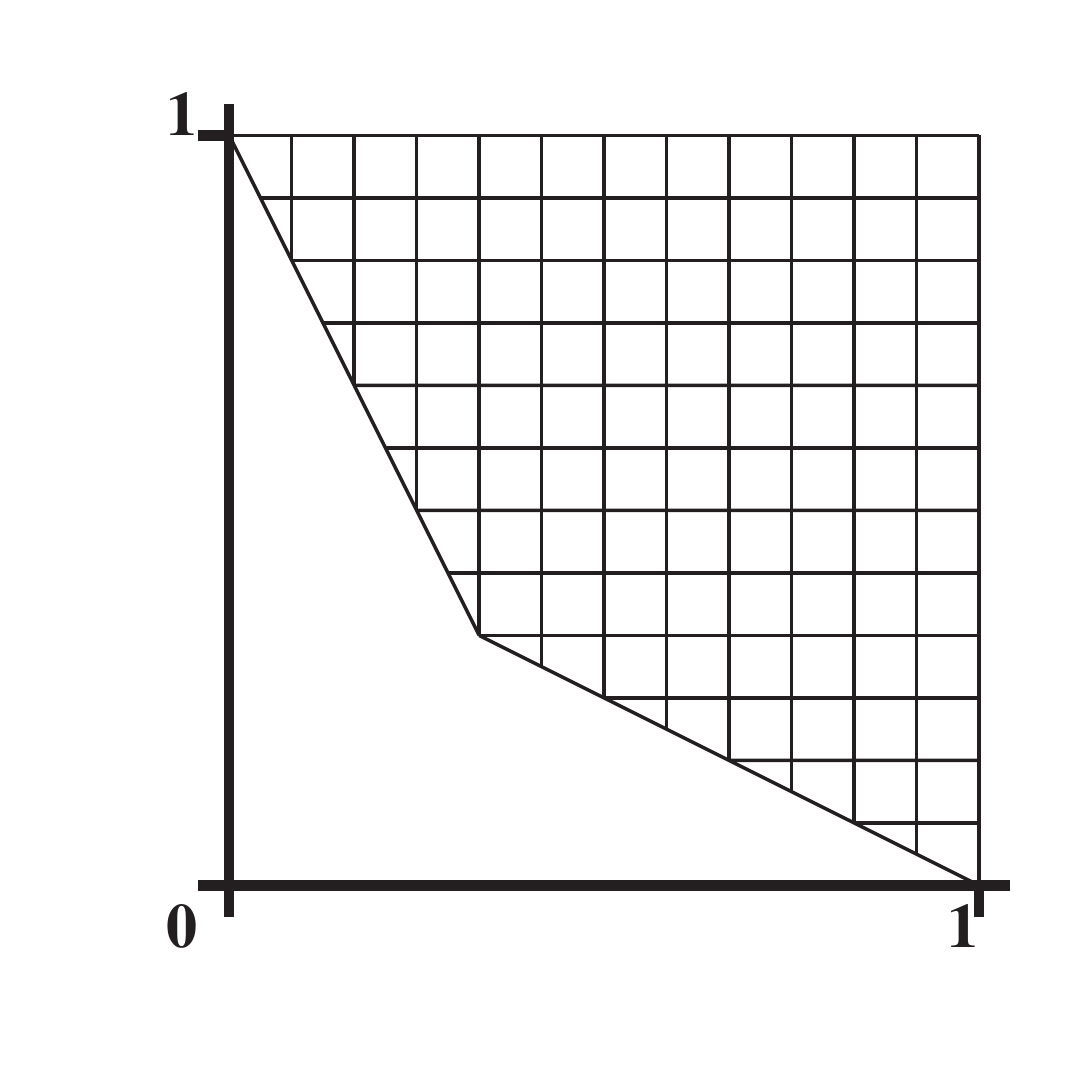} 
   \end{minipage}
    \begin{minipage}{2.7in}
   \includegraphics[width=2.6in]{Sp2sm.pdf} 
   \end{minipage}
   \caption{Two examples of the sets $S_p$; the one on the right shows
   the special   case where $p=0.5$.}
   \label{fig:Sp}
\end{figure}

Let $S_p\=S_{\mu_p}$ be the corresponding subset of $\oii$.
If $0<p<1$, $\mu_p$ has the distribution function
\begin{equation}
  F_{\mu_p}(x)=
  \begin{cases}
\frac{1-p}{p}x, & 0\le x\le p,	
\\
1-\frac{p}{1-p}(1-x), & p\le x\le 1,	
  \end{cases}
\end{equation}
and it follows from \eqref{smu1} that $S_p$ is the quadrilateral with
vertices  $(0,1)$, $(1-p,1-p)$, $(1,0)$ and $(1,1)$,
see \refF{fig:Sp}.
In the special case $p=1/2$, $\mu_p$ is the
uniform distribution on $\oi$, and $S_p$ is the triangle
\set{(x,y)\in\oii:x+y\ge1} pictured in \refF{fig:sps2}
 with vertices  $(0,1)$, $(1,0)$ and $(1,1)$.
Finally, $S_0$ consists of the upper and right edges only, and $S_1=\oii$.

Removing any vertex from $\tnp$ (and relabeling the remaining ones)
yields $\tnxp{n-1}$. It follows that the same property holds for
$\tnpx$, so $\tnpx$ satisfies the assumptions of \refT{Trand}.
Since $\tnpx$ has the same degree distribution as $\tnp$, Theorems
\refand{Trand}{Ttnp} show the following equality.

\begin{corollary}\label{C=}
  If $0\le p\le 1$ and $n\ge1$, then $\tnpx\eqd \tsxn{S_p}$.
\end{corollary}

Hence the random threshold graphs in this subsection are special cases
of the general construction in \refSS{SSset}. We can also
construct them using random weights as in \refSS{SSweights}.

\begin{corollary}\label{C=2}
  If $0\le p\le 1$ and $n\ge1$, then $\tnpx\eqd \txxn{0}$,
where $X$ has the density $1-p$ on $(-1,0)$ and
$p$ on $(0,1)$. 
\end{corollary}

\begin{proof}
  A simple calculation shows that the set $S$ given by
  \eqref{jeppe} is the quadrilateral $S_p$.
\end{proof}

We may transform $X$ by a linear map; for example, we may equivalently
take $X$ with density $2(1-p)$ on $(0,1/2)$ and
$2p$ on $(1/2,1)$, with the threshold $\taux=1$.
In particular, $\tnxx{1/2}\eqd\txxxn{U,1}$, where $U\sim U(0,1)$ as in
\refE{Euniform}.

\subsection{Uniform random \tg{s}}\label{SSuniformlim}

Let $T_n$ be a random unlabeled \tg{} of order
$n$ with the uniform distribution studied in
\refS{Suniform}.
Similarly, let $\tln$ be a random labeled \tg{} of
order $n$ with the uniform distribution.
Although $T_n$ and $\tln$ have different distributions,
see \refS{Suniform}, the next theorem shows that they have the same limit 
as \ntoo.

\begin{theorem}\label{Tdegs}
  The degree distributions $\nu(T_n)$ and
  $\nu(\tln)$ both converge in probability to the uniform
  distribution $\gl$ on $\oi$. Hence,
  $T_n\pto\gG_\gl$ and $\tln\pto\gG_\gl$.
\end{theorem}

By \refSS{SSunlabeled},  $T_n\eqd T_{n,1/2}$;
hence the result for unlabeled graphs follows from \refT{Ttnp}.

\begin{proof}
  We use \refT{Tblock}; in fact, the proof works for random
  \tg{s} generated by \refAlg{Alg3} for any \iid{} random variables
  $B_2,B_3,\dots$ with finite mean, and any $B_1$. (In the case when
  $B_2$ is always  a 
  multiple of some $d>1$, there is a trivial modification.)
Let $\gb\=\E B_2$.

The algorithm starts by choosing (random) block lengths
$B_1,B_2,\dots$ until their sum is at least $n$, and
then rejects them and restarts (Step 3) unless the sums is exactly
$n$. It is simpler to ignore this check, so we consider the
following modified algorithm: Take  $B_1,B_2,\dots$ as above.
Let $S_k\=\sum_{j=1}^kB_j$ be their partial sums and
let $\elm(n)\=\min\set{k:S_k\ge n}$. Toss a
coin to determine whether the first block is isolated or dominating,
and construct a random \tg{} 
by adding $\elm(n)$ blocks of vertices with $B_1,
\dots,B_{\elm(n)}$ elements, alternatingly isolated and
dominant.

This gives a random graph $\tgn$ with $S_{\elm(n)}$ vertices,
but conditioned on $S\nun=n$, we obtain the desired
random threshold graph. (Cf.\ \refT{Tblock}.) Since 
$\P(S\nun=n)$ converges to $1/\gb>0$ by
renewal theory, it suffices to prove that
$\nu(\tgn)\pto\gl$ as \ntoo.
In fact, we will show that $\nu(\tgn)\asto\gl$ if we first choose an
infinite sequence $B_1,B_2,\dots$ and then let \ntoo.

Let $\so_m\=\sum_{2k+1\le m} B_{2k+1}$
and $\se_m\=\sum_{2k\le m} B_{2k}$ be
the partial sums of the odd and even terms. By the law of large
numbers, \as{} $S_n/n\to\gb$ and
$\so_n/n\to\frac12\gb$, $\se_n/n\to\frac12\gb$.
We now consider a fixed sequence $(B_j)_1^\infty$ such that these limits hold.
Since $S_{\elm(n)-1}<n\le S\nun$, it follows, as
is well-known, that $n/\elm(n)\to\gb$, so
$\elm(n)=n/\gb+o(n)$.

Suppose for definiteness that the first block is chosen to be
isolated; then every odd block is isolated and every even block is
dominating. (In the opposite case, interchange even and odd below.)
If $i\in(S_{2k},S_{2k+1}]$, then $i$
belongs to block $2k+1$, so $i$ is added as isolated, and
the neighbors of $i$ will be only the vertices added after
$i$ as dominating, \ie{}
$\bigcup_{k<\ell\le\elm(n)/2}(S_{2\ell-1},S_{2\ell}]$, 
and 
\begin{equation*}
  d(i)=\sum_{2k<2\ell\le\elm(n)} B_{2\ell} 
=\se\nun-\se\nui.
\end{equation*}
If instead $i\in(S_{2k-1},S_{2k}]$, then
$i$ is also joined to all vertices up to $S_{2k}$,
and thus
\begin{equation*}
  d(i)
=\sum_{2\ell\le\elm(n)} B_{2\ell} +\sum_{2\ell+1\le\elm(i)} B_{2\ell+1} 
=\se\nun+\so\nui.
\end{equation*}
Hence, if $i$ is in an odd block,
\begin{equation*}
  \frac{d(i)}n
=
\frac1n\Bigpar{\elm(n)\frac\gb2-\elm(i)\frac\gb2+o(n)}
= \frac{n-i+o(n)}{2n}
=\frac12-\frac{i}{2n}+o(1),
\end{equation*}
and if $i$ is in an even block, similarly,
\begin{equation*}
  \frac{d(i)}n
=\frac12+\frac{i}{2n}+o(1).
\end{equation*}
Now fix $t\in(0,1/2)$ and let $\eps>0$. Then the
following holds if $n$ is large enough:
If $i$ is in an even block, then $d(i)/n\ge 1/2+o(1)>t$.
If $i$ is in an odd block and $i\le i_1\=(1-2t-2\eps)n$, then 
$d(i)/n=\frac12(n-i)/n+o(1)\ge t+\eps+o(1)>t$. 
If $i$ is in an odd block and $i\ge i_2\=(1-2t+2\eps)n$, then 
$d(i)/n=\frac12(n-i)/n+o(1)\le t-\eps+o(1)<t$. 
Consequently, for large $n$, $d(i)/n\le t$ only if
$i$ is in an odd block $(S_{2k},S_{2k+1}]$,
  and in this case $2k+1>\elm(i_1)$ is necessary and
  $2k+1>\elm(i_2)$ is sufficient. Hence,
  \begin{equation*}
\so\nun-\so_{\elm(i_2)}	
\le
|\set{i:d(i)/n\le t}|
\le
\so\nun-\so_{\elm(i_1)}	.
  \end{equation*}
Since $\nu(\tgn)[0,t]=\frac1n
|\set{i:d(i)/n\le t}|$ and
\begin{equation*}
\frac1n\bigpar{\so\nun-\so_{\elm(i_j)}}
=\frac{\gb(\elm(n)-\elm(i_j))+o(n)}{2n}=\frac{n-i_j+o(n)}{2n}
=t\pm\eps+o(1),
\end{equation*}
it follows that
  \begin{equation*}
t-\eps+o(1)\le
\nu(\tgn)[0,t]
\le
t+\eps+o(1).
  \end{equation*}
Since $\eps$ is arbitrary, this shows that
$\nu(\tgn)[0,t]\to t$, for every
$t\in(0,\tfrac12)$.
We clearly obtain the same result if the first block is dominating.

For $t\in(\frac12,1)$ we can argue similarly, now
analysing the dominant blocks. Alternatively, we may apply the result
just obtained to the complement of $\tgn$, which is obtained
from the same $B_j$ by switching the types of the blocks.
This shows that $\nu(\tgn)[0,t]\to t$ for
$t\in(\tfrac12,1)$ too.

Hence,  $\nu(\tgn)[0,t]\to t$ for every
$t\in(0,1)$ except possibly $\frac12$, which shows
that  $\nu(\tgn)\to \gl$.
\end{proof}

\section{Vertex degrees in uniform random threshold graphs}
\label{Sdegrees}

We have seen in \refT{Tdegs} that the normalized degree
distributions $\nu(T_n)$ and $\nu(\tln)$ for
uniform unlabeled and labeled random \tg{s} both converge to
the uniform distribution on $\oi$. This is for weak
convergence of distributions in $\cP$, which is equivalent to
averaging over degrees in intervals $(an,bn)$; we here refine this by
studying individual degrees.

Let $N_d(G)$ be the number of vertices of degree $d$ in
the graph $G$. Thus, $D_G$, the degree of a random vertex
in $G$ has distribution $\P(D_G=d)=N_d/v(G)$.
(Recall that $\nu(G)$ is the distribution of
$D_G/v(G)$, see \refS{Sdeg}.)

We will study the random variables $N_d(T_n)$ and
$N_d(\tln)$ describing the numbers of vertices of a given
degree $d$ in a uniform random unlabeled or labeled \tg, and
in particular their expectations $\E N_d(T_n)$ and $\E N_d(\tln)$;
note that $\E N_d(T_n)/n$ and $\E N_d(\tln)/n$
are the probabilities that a given (or random) vertex in the random
graph $T_n$ or $\tln$ has degree $d$.
By symmetry under complementation,
\begin{equation*}
  N_d(T_n)\eqd N_{n-1-d}(T_n)
\qquad\text{and}\qquad
  N_d(\tln)\eqd N_{n-1-d}(\tln).
\end{equation*}

Let us first look at $N_0$, the number of isolated vertices. 
(By symmetry, we have the same results for $N_{n-1}$, the
number of dominating vertices). 
Note that, for every $n\ge2$,
$\P(N_0(T_n)=0)=\P(N_0(\tln)=0)=1/2$ by symmetry.

\begin{theorem}
  \label{TQ1}
  \begin{thmenumerate}
\item
For any $n\ge1$,
\begin{equation}
  \label{tq1u}
\P\bigpar{N_0(T_n)=j} =
  \begin{cases}
	2^{-j-1}, & 0\le j\le n-2, \\
0, & j=n-1, \\
2^{-n+1}, & j=n.
  \end{cases}
\end{equation}
In other words, if $X\sim\Ge(1/2)$, then
$N_0(T_n)\eqd X_n'$, where $X_n'\=X_n$ if
$x<n-1$ and $X'_n\=n$ if $X_n\ge n-1$. 
Furthermore, $\E N_0(T_n)=1$, and
$N_0(T_n)\dto\Ge(1/2)$ as \ntoo, with convergence of all moments.
\item
$\P\bigpar{N_0(\tln)=j}=t(n,j)/t(n)$, where
$t(n,j)$ is given by \eqref{tnj}; in particular, if\/
$0\le j\le n-2$, then
\begin{equation*}
  \P\bigpar{N_0(\tln)=j}=
\frac1{2j!}\frac{t(n-j)/(n-j)!}{t(n)/n!}
=\frac1{2j!}(\log 2)^j\bigpar{1+O\xpar{\rho^{n-j}}}
\end{equation*}
with $\rho=\log2/(2\pi)\approx0.11$.
Hence, $N_0(\tln)\dto\Po(\log 2)$  
as \ntoo{}
with convergence of all moments; in particular,
$\E N_0(\tln)\to\log 2$. 
  \end{thmenumerate}
\end{theorem}

\begin{proof}
\pfitem{i}
  A \tg{} has $j$ isolated vertices if and only if the
  extended binary code $\ga_1\dotsm\ga_n$ in
  \refS{Suniform} ends with exactly $j$ 0's.
For a random unlabeled \tg{} $T_n$, the binary code
  $\ga_2\dotsm\ga_n$ is uniformly distributed, and thus
  \eqref{tq1u} follows.
The remaining assertions follow directly.

\pfitem{ii}
In the labeled case, the exact distribution is given by
\eqref{tnj}, and the asymptotics follow by
\eqref{tn}.
Uniform integrabilit of any power $N_0(\tln)^m$ follows by
the same estimates, and thus moment convergence holds.
\end{proof}

For higher degrees, we begin with an exact result for the unlabeled
case.

\begin{theorem}  \label{TEdegrees}
 $ \E N_d(T_n)=1$ for every $d=0,\dots,n-1$.
\end{theorem}

Actually, this is the special case $p=1/2$ of a more general theorem for the
random \tg{} $\tnp$ defined in \refSS{SSaddition}:
(Cf.\ \refT{Ttnp}, which is for weak convergence, but on the
other hand yields an \as{} limit while we here study the expectations.)
\begin{theorem}
  \label{TEdegreesp}
Let $0<p<1$.
If $q=1-p$ and $X\sim\Bin(n,p)$, then, for
$0\le d\le n-1$,
\begin{equation*}
\E N_d(\tnp)=\frac qp+\Bigpar{\frac pq - \frac qp}\P(X\le d).
\end{equation*}
\end{theorem}

\begin{proof}
We use the definition in \refSS{SSaddition}. (For
the uniform case
  $p=1/2$, this is \refAlg{Alg1}.) 
  Let $d_i$ be the degree of vertex $i$. Then, if
  $\ga_1\dotsm\ga_n$ is the extended binary code of the graph,
we have
  \begin{equation*}
d_i=(i-1)\ga_i+\sum_{j=i+1}^n\ga_j.	
  \end{equation*}
Since the $\ga_i$ are \iid{} $\Be(p)$ for
$i=2,\dots,n$, the probability generating function of
$d_i$ is
\begin{equation*}
  \E x^{d_i} = \E
  x^{(i-1)\ga_i}\prod_{j=i+1}^n\E x^{\ga_j}
=(p x^{i-1}+q)(px+q)^{n-i}.
\end{equation*}
Consequently,
\begin{equation*}
  \begin{split}
\sum_{d}\E N_d(\tnp)x^d
&=\sum_{i=1}^n\E x^{d_i}
=\sum_{i=1}^npx^{i-1}(px+q)^{n-i}+\sum_{i=1}^nq(px+q)^{n-i}
\\&
=p\frac{x^n-(px+q)^n}{x-(px+q)}+q\frac{1-(px+q)^n}{1-(px+q)}
\\&
=\frac{(q/p)+(p/q-q/p)(px+q)^n-(p/q)x^n}{1-x}.
  \end{split}
\end{equation*}
In the special case $p=1/2$, this is
$(1-x^n)/(1-x)=\sum_{d=0}^{n-1}x^d$, which shows
\refT{TEdegrees} by identifying coefficients.
For general $p$, \refT{TEdegreesp} follows in the same way.
\end{proof}

Recall that $R_d$ denotes the number of preferential
arrangements, or surjection numbers, given in \eqref{rn}.
\begin{theorem}\label{Tend}
  \begin{thmenumerate}
\item
In the unlabeled case, for any sequence $d=d(n)$ with
$0\le d\le n-1$,
$N_d(T_n)\dto\Ge(1/2)$ with convergence of all moments.	
\item
In the labeled case,
let $X_d$, $0\le d\le\infty$, have the modified
Poisson distribution given by
\begin{equation*}
  \P(X_d=\ell)=
  \begin{cases}
\frac{\gamd}{\log2}\P\bigpar{\Po(\log2)=\ell}
 =\gamd\frac{(\log2)^{\ell-1}}{2\cdot\ell!}, & \ell\ge1, 
\\
1-\frac{\gamd}{2\log2},&\ell=0,
  \end{cases}
\end{equation*}
where $\gam_0\=\log 2$, $\gamd\=2R_d(\log2)^{d+1}/d!$ 
for $d\ge1$, and $\gamoo\=1$.
Then, for every fixed $d\ge0$, 
$N_d(\tln)\eqd N_{n-1-d}(\tln)\dto X_d$,
and for every sequence $d=d(n)\to\infty$ with
$n-d\to\infty$, $N_d(\tln)\eqd
N_{n-1-d}(\tln)\dto X_\infty$ as \ntoo, in both cases with
convergence of all moments.

In particular, $\E N_d(\tln)=\E
N_{n-1-d}(\tln)$ converges to $\gamd$ for every
fixed $d$, and to $\gamoo=1$ if\/
$d\to\infty$ and $n-d\to\infty$.
  \end{thmenumerate}
\end{theorem}

In the labeled case we thus have, in particular, 
$\E N_0(\tln)\to\log2\approx0.69315$,
$\E N_1(\tln)\to2(\log2)^2\approx0.96091$,
$\E N_2(\tln)\to3(\log2)^3\approx0.99907$,
$\E N_3(\tln)\to\frac{13}3(\log2)^4\approx1.00028$.
The values for degrees 0 and 1 (and symmetrically $n-1$ and
$n-2$) are thus substantially smaller than 1, which is clearly
seen in \refF{fig:exgr5}. 
(We can regard this as an edge effect; the vertices with degrees close
to 0 or $n-1$ are the ones added last in
\refAlg{Alg3}. 
\refF{fig:exgr5} also shows an edge effect at the other side; there is
a small bump for degrees arond $n/2$, which correspond to the
vertices added very early in the algorithm; this bump vanishes
asymptotically, as shown by \refT{Tend}; we believe that it
has height of order $n\qqw$ and width of order
$n\qq$, but we have not analyzed it in detail.)

\begin{proof}
  The cases $d=0$ and $d=n-1$ follow from
  \refT{TQ1}. We may thus suppose $1\le d\le n-2$.
We use \refAlg{Alg3}.
We know that vertices in each block have the same degree, while
  different blocks have different degrees; thus there is at most one
  block with degrees $d$.

Let $p_d(\ell)$ be the probability that there is such a block
of length $\ell\ge1$, and that this block is added as
isolated. By symmetry, the probability that there is a dominating
block of length $\ell$ with degrees $d$ is
$p_{n-1-d}$ and thus
\begin{equation}
  \label{p2}
\P(N_d=\ell)=p_d(\ell)+p_{n-1-d}(\ell),
\qquad \ell\ge1.
\end{equation}

If block $j$ is an isolated block, then the degree of the
vertices in it equals the number of vertices added as dominating after
it, \ie,
$B_{j+1}+B_{j+3}+\dots+B_{j+2k-1}$, if the
total number $\tau$ of blocks is $j+2k-1$ or $j+2k$.
Consequently, there is an isolated block of length $\ell$
with vertices of degree $d$ if and only if there exist
$j\ge 1$ and $k\ge1$ with
\begin{itemize}
  \item $B_j=\ell$,
\item
block $j$ is isolated,
\item
$\sum_{i=1}^k B_{j+2i-1}=d$,
\item
$\sum_{i=1}^{j+2k-1} B_{i}=n$ or $\sum_{i=1}^{j+2k} B_{i}=n$.
\end{itemize}
Recall that $B_1,B_2,\dots$ are independent and that
$B_2,B_3,\dots$ have the same distribution while $B_1$
has a different one. (The distributions differ between the unlabeled
and labeled cases.) Let
\begin{equation*}
  \hS_n\=\sum_{i=1}^m B_i
\qquad\text{and}\qquad
  S_n\=\sum_{i=1}^m B_{i+1}.
\end{equation*}
Further, let 
\begin{align*}
  \hu(n)&=\sum_{m=0}^\infty  \P(\hS_m=n)
=\P(B_\tau=n)=\P\Bigpar{\sum_{i=1}^\tau  B_i=n},
\\
u(n)&=\sum_{m=0}^\infty  \P(S_m=n),
\end{align*}
and recall that $u(n),\hu(n)\to1/\mu\=1/\E B_2$
(exponentially fast) by standard renewal theory (for example by
considering generating functions).
For any $m\ge j+2k-1$,
\begin{equation*}
  \sum_{i=1}^m B_i - B_j - \sum_{i=1}^k  B_{j+2i-1}
\eqd
  \begin{cases}
S_{m-1-k}, & j=1,\\
\hS_{m-1-k}, & j\ge2,
  \end{cases}
\end{equation*}
and it follows that, since $B_1,B_2,\dots$ are
independent and we condition on $\hS_\tau=n$,
\begin{multline}
  \label{p4}
p_d(\ell)=
\frac{1}{2\hu(n)}	
\Bigl\{\sum_{k=1}^\infty\P(B_1=\ell)\P(S_k=d)
\\\shoveright{\cdot\bigpar{\P(S_{k-1}=n-\ell-d)+\P(S_k=n-\ell-d)}}
\\
\shoveleft{\quad\qquad+\sum_{j=2}^\infty\sum_{k=1}^\infty\P(B_j=\ell)\P(S_k=d)}
\\
\cdot\bigpar{\P(\hS_{j+k-2}=n-\ell-d)+\P(\hS_{j+k-1}=n-\ell-d)}
\Bigr\}
\end{multline}
In the double sum, $\P(B_j=\ell)=\P(B_2=\ell)$ does
not depend on $j$, so the sum is at most 
\begin{equation*}
  \begin{split}
\P(B_2=\ell)\sum_k\P(S_k=d)2\hu(n-\ell-d)
&=2\P(B_2=\ell)u(d)\hu(n-\ell-d)
\\&=O(\P(B_2=\ell)).	
  \end{split}
\end{equation*}
Similarly, the first sum is
$O(\P(B_1=\ell))=O(\P(B_2=\ell))$, and it follows
that 
$p_d(\ell)=O(\P(B_2=\ell))$ and  thus, by \eqref{p2},
  \begin{equation}
	\label{mara}
  \P(N_d=\ell)=O(\P(B_2=\ell)),
  \end{equation}
uniformly in $n$, $d$ and $\ell$. This shows
tightness, so convergence
$\P(N_d=\ell)\to\P(X=\ell)$ for some
non-negative integer valued random variable $X$ and each fixed
$\ell\ge1$ implies convergence in distribution (\ie,
for $\ell=0$ too). Further, since all moments of $B_2$
are finite, \eqref{mara} implies that all moments $\E
N_d^m$ are bounded, uniformly in $d$ and $n$; hence
convergence in distribution implies that all moments converge too.
In the rest of the proof we thus let $\ell\ge1$ be fixed.

If $d\le n/2$ it is easy to see that 
$\P(S_{k-1}=n-\ell-d)+\P(S_{k}=n-\ell-d)=O\bigpar{(n-\ell-d)\qqw}=O(n\qqw)$,
uniformly in $k$,
so the first sum in \eqref{p4} is $O\bigpar{n\qqw u(d)}=O\bigpar{n\qqw}$.
If $d>n/2$, we similarly have $\P(S_k=d)=O\bigpar{d\qqw}=O\bigpar{n\qqw}$
and thus the sum is 
$O\bigpar{n\qqw u(n-\ell-d)}=O\bigpar{n\qqw}$.
Hence \eqref{p4} yields
\begin{multline}\label{p4b}
  p_d(\ell)=O\bigpar{n\qqw}+\frac{\P(B_2=\ell)}{2\hu(n)}
\sum_{k=1}^\infty\P(S_k=d)
\\\cdot
\Bigpar{\sum_{i=k}\P(\hS_i=n-\ell-d)+\sum_{i=k+1}\P(\hS_i=n-\ell-d)}.
 \end{multline}
The term with $i=k$ can be taken twice, just as the ones with
$i>k$,
since
$\sum_k\P(S_k=d)\P(\hS_k=n-\ell-d)=\Onqqw$ by
the same argument as for the first sum in \eqref{p4}.
Further, for $i\ge k$, $\hS_i-\hS_k\eqd
S_{i-k}$ and is independent of $\hS_k$; thus
\begin{equation*}
  \P(\hS_i=n-\ell-d)=\P(S_{i-k}=n-\ell-d-\hS_k)
=
\E\P\bigpar{S_{i-k}=n-\ell-d-\hS_k\mid\hS_k}
\end{equation*}
and
$\sum_{i=k}^\infty \P(\hS_i=n-\ell-d)=\E u(n-\ell-d-\hS_k)$.
Hence, \eqref{p4b} yields
\begin{equation}
  \label{p5}
p_d(\ell)
=\frac{\P(B_2=\ell)}{\hu(n)}
\sum_{k=1}^\infty \P(S_k=d)\E u(n-\ell-d-\hS_k)+\Onqqw.
\end{equation}

If $d$ is fixed, then $\E
u(n-\ell-d-\hS_k)\to\mu\qw$ by dominated convergence
as \ntoo{} for each $k$, and thus \eqref{p5}
yields, by dominated convergence again,
\begin{equation}\label{pdl}
p_d(\ell)
\to \P(B_2=\ell)
\sum_{k=1}^\infty \P(S_k=d)
=u(d) \P(B_2=\ell).
\end{equation}

If $d\to\infty$, 
we use the
fact that $u(m)-\ett{m\ge0}\mu\qw$ is summable
over $\bbZ$ to see that
\begin{equation*}
  \E  u(n-\ell-d-\hS_k)-\mu\qw\P(n-\ell-d-\hS_k\ge0)
=O\bigpar{\max_m\P(\hS_k=m)},
\end{equation*}
which tends to 0 as $k\to\infty$; on the other hand,
$\P(S_k=d)\to0$ for every fixed $k$.
It follows that \eqref{p5} yields
\begin{equation*}
p_d(\ell)
= \P(B_2=\ell)
\sum_{k=1}^\infty \P(S_k=d)\P(\hS_k\le n-\ell-d)+o(1).
\end{equation*}
If $\tau_d\=\min\set{k:S_k\ge d}$, and
$\hS'_k$ denotes a copy of $\hS_k$ independent of $\set{S_j}_1^\infty$,
then
\begin{equation*}
\sum_{k=1}^\infty \P(S_k=d)\P(\hS_k\le n-\ell-d)
=u(d)\P\bigpar{\hS'_{\tau_d}\le n-\ell-d\mid S_{\tau_d}=d}.
\end{equation*}
It is easy to see that, with $\gss\=\Var(B_2)$,
as $d\to\infty$,  
\begin{equation*}
\bigpar{(\hS'_{\tau_d}-d)/\sqrt d\mid S_{\tau_d}=d}
=
\bigpar{(\hS'_{\tau_d}-S_{\tau_d})/\sqrt d\mid S_{\tau_d}=d}
\dto N(0,2\gss/\mu),
\end{equation*}
\cf{} \cite{SJ28} (the extra conditioning on
$S_{\tau_d}=d$ makes no difference).
Hence, when $d\to\infty$,
\begin{equation*}
  \begin{split}
p_d(\ell)
&= \P(B_2=\ell)u(d)\Phi\bigpar{(n-\ell-2d)/\sqrt d}
+o(1).
  \end{split}
\end{equation*}
(By \eqref{pdl}, this holds for fixed $d$ too.)
We next observe that 
$\Phi\bigpar{(n-\ell-2d)/\sqrt d}=\Phi\bigpar{(n-2d)/\sqrt{
	n/2}}+o(1)$; this is easily seen by considering separately the three
cases $d/n\to a\in[0,1/2)$, $d/n\to a\in(1/2,1]$,  and
$d/n\to 1/2$ and $(n-2d)/\sqrt{n/2}\to b\in[-\infty,\infty]$
(the general case follows by considering suitable subsequences).
Hence, we have when $d\to\infty$, recalling that then $u(d)\to\mu\qw$,
\begin{equation*}
  \begin{split}
p_d(\ell)
= \mu\qw\P(B_2=\ell)\Phi\bigpar{(n-2d)/\sqrt {n/2}}+o(1).
  \end{split}
\end{equation*}

For fixed $d$, this implies that
$p_{n-d-1}(\ell)\to0$, and thus \eqref{p2} and
\eqref{pdl} yield
\begin{equation*}
  \begin{split}
\P(N_d=\ell)
&=
p_d(\ell)+p_{n-1-d}(\ell)
=u(d)\P(B_2=\ell)+o(1).
  \end{split}
\end{equation*}

Similarly, if $d\to\infty$ and
$n-d\to\infty$, 
\begin{equation*}
  \begin{split}
\P&(N_d=\ell)
=
p_d(\ell)+p_{n-1-d}(\ell)
\\&
= \mu\qw\P(B_2=\ell)
\bigpar{
\Phi\bigpar{(n-2d)/\sqrt{n/2}}
+
\Phi\bigpar{(2d+2-n)/\sqrt{n/2}}}
+o(1)
\\&
=\mu\qw\P(B_2=\ell)+o(1).
  \end{split}
\end{equation*}

We have thus proven convergence as \ntoo, with all moments, 
$N_d\dto X_d$ for fixed $d$
and $N_d\dto X_\infty$ for $d=d(n)\to\infty$
with $n-d\to\infty$, where 
\begin{align}
\P(X_d=\ell)&=u(d)\P(B_2=\ell)
=2u(d)\P(B^*=\ell)
, \qquad  \ell\ge1, \label{p7a}\\
\P(X_d=0)&=1-\P(X_d\ge1)=1-u(d), \label{p7b}
\end{align}
for $1\le d\le\infty$, 
with $u(\infty)\=\mu\qw$.

In the unlabeled case, $B_2=(B^*\mid B^*\ge1)\eqd B^*+1$
with $B^*\sim\Ge(1/2)$.
Consider a random infinite string $\ga_1\ga_2\dotsm$ of
\iid{} $\Be(1/2)$ binary digits, and define a block
as a string of $m\ge0$ 0's followed by a single 1. Then
$B_{j+1}$, $j\ge1$, can be interpreted as the
successive block lengths in $\ga_1\ga_2\dotsm$, and
thus $u(d)$ is the probability that some block ends at
$d$,
\ie, $u(d)=\P(\ga_d=1)=1/2$, for every
$d\ge1$. It follows from \eqref{p7a}--\eqref{p7b}
that $X_d\eqd B^*\sim\Ge(1/2)$ for every
$d\ge1$, and (i) follows.

In the labeled case, when $B^*\sim\Po(\log2)$, we use generating functions:
\begin{equation*}
  \begin{split}
  \sum_{d=0}^\infty u(d)x^d 
&= \sum_{k=0}^\infty \E x^{S_k}
= \sum_{k=0}^\infty \bigpar{\E x^{B_2}}^k
=\frac1{1-\E x^{B_2}}	
=\frac{\P(B^*\ge1)}{1-\E x^{B^*}}	
\\&
=\frac{1/2}{1-e^{(x-1)\log2}}
=\frac1{2-e^{x\log2}}
=\sum_{d=0}^\infty \frac{R_d}{d!}(x\log2)^d,
  \end{split}
\end{equation*}
where we recognize the gererating function \eqref{rgen}.
Thus, $u(d)=R_d(\log2)^{d}/d!$. (A direct combinatorial
proof of this is also easy.)

We let, using $\mu\=\E B_2=\E B^*/\P(B^*\ge1)=2\log2$,
$$
\gamd\=\E X_d=u(d)\E B_2=\mu u(d)=2\log 2 u(d)=2R_d(\log2)^{d+1}/d!
$$
and note that $\gamd\to\gamoo=1$ as
$d\to\infty$ since $u(d)\to\mu\qw$, or by
the known asymptotics of $R_d$ \cite[(II.16)]{FS}.
The description of $X_d$ in the statement now follows from
\eqref{p7a}--\eqref{p7b}. 
\end{proof}

\section{Random bipartite threshold graphs}\label{SrandomB} 

The constructions and results in \refS{Srandom} have analogues for
\btg{s}. The proofs are simple modifications of the ones above and are
omitted.

\subsection{Increasing set}\label{SSsetB}
  For any  increasing $S\subseteq\oii$, 
  define $\tsnn\=G(n_1,n_2,\etta_S)$.
In other words,  take \iid{} random variables
  $U_1',\dots,U_{n_1}',\allowbreak U_1'',\dots,U_{n_2}''\sim U(0,1)$  
and  draw an edge $ij$ if $(U_i',U_j'')\in S$.

  \begin{theorem}\label{TtsnB}
	As \nntoo, $\tsnn\asto\gG_S''$. In particular,
the degree distribution
$\nu_1(\tsn)\asto\nu_1(\ggb_S)$, which equals the
distribution of $\gf_S(U)$ defined by \eqref{ttsn}.
\nopf
  \end{theorem}

As in \refS{Srandom}, this gives a canonical  representation of random
\btg{s} under natural assumptions.

\begin{theorem}\label{TrandB}
  Suppose that $(\Gnn)_{n_1,n_2\ge1}$ are random bipartite threshold
  graphs with   $V_1(\Gnn)=[n_1]$ and 
  $V_2(\Gnn)=[n_2]$ 
such that the distribution of each $\Gnn$ is invariant
  under permutations of $V_1$ and $V_2$ and that the restrictions
(induced subgraphs)
of $G_{n_1+1,n_2}$  and  $G_{n_1,n_2+1}$ 
to  $V(G)$ both have the same distribution as $\Gnn$, for every $n_1,n_2\ge1$.
If further $\nu_1(\Gnn)\pto\mu$ as \nntoo, for some $\mu\in\cP$, then,
for every
  $n_1,n_2$, $\Gnn\eqd\tsxnn{\smu}$.
\nopf
\end{theorem}

\subsection{Random weights}\label{SSweightsB}
Definition \ref{B1.1} suggests the following construction:
\begin{eqenumerate}
\item\label{DRBTG1}
Let $X$ and $Y$ be two random variables and
let $\taux\in\bbR$. Let
$X_1,X_2,\dots,$ be  copies of $X$
and $Y_1,Y_2,\dots,$  copies of $Y$, all independent,
and 
let $\txytnn$ be the bipartite threshold graph with vertex sets
$[n_1]$ and $[n_2]$ and edges $ij$
for all pairs $ij$ such that $X_i+Y_j>\taux$.
  \end{eqenumerate}

\begin{theorem}\label{TkonstB}
Let $S$ be the  increasing set 
  \begin{equation}\label{jeppeB}
	S\=\set{(x,y)\in(0,1]^2:F_X\qw(x)+F_Y\qw(y)>\taux}.
  \end{equation}
Then $\txytnn\eqd\tsnn$ for every $n_1,n_2\ge1$.

Furthermore, as \nntoo, the degree distribution
$\nu_1(\txytnn)\asto\mu$ and thus
$\txytnn\asto\ggb_\mu$, where
$\mu\in\cP$ is the distribution of the random variable
$1-F_Y(\taux-X)$, \ie{} 
\begin{equation}\label{e2fmub}
 \mu[0,s]=\P\bigpar{1-F_Y(\taux-X)\le s},
\qquad s\in\oi.
\end{equation}
\nopf
\end{theorem}
In the special case when $\P(X\in\oi)=1$, $Y\sim U(0,1)$ and
$\taux=1$, \eqref{e2fmub} yields $\mu[0,s]=\P(X\le s)$, so $\mu$ is
the distribution of $X$; further, the set $S$ in \eqref{jeppeB} is
\aex{} equal to $\smu$ in \eqref{smu2}.

\begin{corollary}
  \label{CkonstB}
If $\mu\in\cps$, let $X$ have distribution $\mu$ and let $Y\sim
U(0,1)$. Then $\txytnn\eqd\tsxnn{\smu}$ for every
$n_1,n_2\ge1$. Furthermore, as \nntoo, $\nu_1(\txytnn)\pto\mu$ and
$\txytnn\pto\ggb_\mu$. 
\end{corollary}

This yields another canonical construction for every $\mu\in\cP$.
(We claim only convergence in probability in \refC{CkonstB};
convergence \as{} holds at least along every increasing subsequence
$(n_1(m),n_2(m))$, see \cite[Remark 8.2]{DJ}.)

\subsection{Random addition of vertices}\label{SSadditionB}

Definition \ref{B1.2} suggests the following construction:
\begin{eqenumerate}
\item\label{DRBTG2}
Let $\tnnpp$ be the random bipartite threshold graph with $n_1+n_2$
vertices 
obtained as follows: Take $n_1$ `white' vertices and $n_2$ `black'
vertices, and arrange them in random order. Then, join each white
vertex with probability $p_1$ to all earlier black vertices, and join
each black vertex with probability $p_2$ to all earlier white vertices
(otherwise, the vertex is joined to no earlier vertex), the decisions
being made independently by tossing a biased coin once for each white
vertex, and another biased coin once for each black vertex.
  \end{eqenumerate}

Let, for $p_1,p_2\in\oi$, $\mupp$ be the probability measure in $\cP$
with 
distribution function
\begin{equation}
  F_{\mupp}(x)=
  \begin{cases}
\frac{1-p_1}{p_2}x, & 0\le x< p_2,	
\\
1-\frac{p_1}{1-p_2}(1-x), & p_2\le x< 1.
  \end{cases}
\end{equation}
Hence, $\mupp$ has density $(1-p_1)/p_2$ on $(0,p_2)$ and $p_1/(1-p_2)$
on $(p_2,1)$; if $p_2=0$ there is also a point mass $1-p_1$ at 0, and
if $p_2=1$ there is also a point mass $p_1$ at 1.
It follows from \eqref{smu1} that the corresponding subset
$\spp\=S_{\mupp}$ of $\oi^2$ is the quadrilateral with
vertices  $(0,1)$, $(1-p_1,1-p_2)$, $(1,0)$ and $(1,1)$ (including
degenerate cases when $p_1$ or $p_2$ is 0 or 1).

This is an extension of the definitions in
\refSS{SSaddition}; we have $\mu_{p,p}=\mu_p$ and $S_{p,p}=S_p$.
Note also that $\mupp\sss=\mu_{p_2,p_1}$. In particular,
$\mupp\in\cps$ only if $p_1=p_2$.

\begin{theorem}\label{TtnpB}
As \nntoo, the degree distributions $\nu_1(\tnnpp)\pto\mupp$  and
$\nu_2(\tnnpp)\pto\muppz$; consequently,
$\tnnpp\pto \ggbpp\=\ggb_{\mupp}\in\ctboo$. 
\end{theorem}

\begin{corollary}\label{C=B}
  If $p_1,p_2\in\oi$ and $n_1,n_2\ge1$, then 
  \begin{equation*}
	\tnnpp\eqd \tsxnn{\spp}\eqd\tnnx{X_1,X_2,0},
  \end{equation*}
where $X_j$ has the density $1-p_j$ on $(-1,0)$ and
$p_j$ on $(0,1)$, $j=1,2$. 
\end{corollary}

Note that if $p_1+p_2=1$, then $\spp$ is the upper triangle
$S_{1/2}\=\set{(x,y):x+y\ge1}$. Hence the distribution of $\tnnpp$
does not depend on $p_1$ as long as $p_2=1-p_1$. In particular, we may
then choose $p_1=1$ and $p_2=0$.
In this case, Definition \ref{DRBTG2} simplifies as follows.

\begin{eqenumerate}
\item\label{DRBTG2u}
Let $\tnn$ be the random bipartite threshold graph with $n_1+n_2$
vertices 
obtained as follows: Take $n_1$ `white' vertices and $n_2$ `black'
vertices, and arrange them in random order.  Join every white
vertex to every  earlier black vertex.
  \end{eqenumerate}

If $p_1=1$ and $p_2=0$, then further $X_1\eqd U\sim U(0,1)$ and
$X_2\eqd U-1$ in \refC{C=B}. Hence, we have found a number of natural
constructions that yield the same random \btg.
\begin{corollary}\label{CUB}
  If $p_1\in\oi$ and $n_1,n_2\ge1$, then 
  \begin{equation*}
	\tnnx{p_1,1-p_1}\eqd \tnnx{1,0}=\tnn\eqd\tsxnn{S_{1/2}}\eqd\tnnx{U,U,1},
  \end{equation*}
with $U\sim U(0,1)$.
\nopf
\end{corollary}
We will see in the next subsection that this random \btg{} is
uniformly distributed as an unlabeled \btg.

\subsection{Uniform random \btg{s}}\label{SSuniformlimB}

It is easy to see that for every \btg, if we color the vertices in
$V_1$ white and the vertices in $V_2$ black, then there is an ordering
of the vertices such that a white vertex is joined to every earlier
black vertex but not to any later. (For example, if there are weights
as in \ref{B1.1}, order the vertices according to $w_i'$ and $w_j''$,
taking the white vertices first in case of a tie.) This yields a 1--1
correspondence between unlabeled \btg{s} on $n_1+n_2$ vertices and
sequences of $n_1$ white and $n_2$ black balls.
Consequently, the number of  unlabeled \btg{s} is
\begin{equation*}
  |\ctnn|=\binom{n_1+n_2}{n_1},
\qquad n_1,n_2\ge1.
\end{equation*}
Moreover, it follows that 
$\tnn$ is  uniformly distributed
in $\ctnn$; hence \refC{CUB} yields the following:

\begin{theorem}\label{TUB}
The random \btg{s} $\tnn$,
$\tnnx{p_1,1-p_1}$ ($0\le p_1\le 1$),  $\tsxnn{S_{1/2}}$, $\tnnx{U,U,1}$
are all
uniformly distributed, 
regarded as unlabeled \btg{s}.
\nopf
\end{theorem}

We have not studied uniform random labeled \btg{s}.

\section{Spectrum of Threshold Graphs}\label{Sspectrum}

There is a healthy literature on the eigenvalue distribution of the
adjacency matrix for various classes of random graphs. Much of this is
focused on the spectral gap (e.g., most $k$-regular graphs are
Ramanujan \cite{Davidoff}). See Jakobson, Miller, Rivin,
Rudnick \cite{Jakobson} for evidence showing that random $k$-regular
graphs have the same limiting eigenvalue distribution as the Gaussian
orthogonal ensemble. The following results show that random threshold
graphs give a family of examples with highly controlled limiting
spectrum.

There is a tight connection between the degree distribution of
a threshold graph and 
the spectrum of its Laplacian,
see \cite{Merris94,HK96,MR05}.
Recall that the \emph{Laplacian} of a graph $G$, with $V(G)=[n]$, say,
is the $n\times n$ matrix $\cL=D-A$, where $A$ is the adjacency matrix
of $G$ and $D$ is the diagonal matrix with entries
$d_{ii}=d_G(i)$. (Thus $\cL$ is symmetric and has row sums 0.)
It is easily seen that $\langle\cL x,y\rangle=\sum_{ij\in E(G)}
(x_i-x_j)(y_i-y_j)$ for $x,y\in\bbR^n$.
The eigenvalues $\gl_i$ of $\cL$ satisfy $0\le\gl_i\le n$,
$i=1,\dots,n$, and we define the normalized spectral distribution
$\nul\in\cP$ as the empirical distribution of $\set{\gl_i/n}_{i=1}^n$.

For a \tg, it is easily seen that if we order the vertices as in
\ref{D1.2} and \refSS{SSunlabeled}, then for each $i=2,\dots,n$ the
function
\begin{equation*}
  \gf_i(j)\=
  \begin{cases}
	-1, & j<i,\\
i-1, & j=i,\\
0, & j>i.
  \end{cases}
\end{equation*}
is an eigenfunction of $\cL$ with eigenvalue $d(i)$ or $d(i)+1$,
depending on whether $i$ is added as isolated or dominating, \ie,
whether $\ga_i=0$ or 1 in the binary code of the graph.
Together with $\gf_1\=1$  (which is an eigenfunction with eigenvalue 0
for any graph), these form an orthogonal basis of eigenfunctions. The
Laplacian spectrum thus can be written
\begin{equation}
  \label{lsp}
\set{0}\cup\set{d(i)+\ga_i:i=2,\dots,n}.
\end{equation}
In particular, the eigenvalues are all integers.

Moreover, \eqref{lsp} shows that the spectrum $\set{\gl_i}_1^n$ is
closely related to the degree sequence; in particular, asymptotically
they are the same in the sense that if $G_n$ is a sequence of \tg{s}
with $v(G_n)\to\infty$ and $\mu\in\cP$, then 
\begin{equation}\label{nulnu}
  \nul(G_n)\to\mu \iff \nu(G_n)\to\mu.
\end{equation}
(See \cite{HK96} for a detailed comparison of the Laplacian spectrum and
the degree sequence for \tg{s}.) In particular, \refT{Tlim} can be
restated using the spectral distribution:
\begin{theorem}  \label{TlimL}
  Let $G_n$ be a sequence of \tg{s} such that
  $v(G_n)\to\infty$. 
Then $G_n$ converges in $\cuq$ as \ntoo, if and only if 
the spectral distributions $\nul(G_n)$ converge to some distribution
  $\mu$. In this case, $\mu\in\cps$ and $G_n\to\Gamma_\mu$.
\nopf
\end{theorem}

\begin{remark}
  It can be shown that the spectrum and the degree sequence
are
  asymptotically close in the sense that \eqref{nulnu} holds for any
  graphs $G_n$ with $v(G_n)\to\infty$, even though  in general there is
  no simple relation like \eqref{lsp}.
\end{remark}

Another relation between the spectrum and the degree sequence for a
\tg{} is that  their Ferrers diagrams are transposes of each other,
see \cite{Merris94, MR05}; this is easily verified from
\eqref{lsp} by induction.
If we scale the Ferrers diagrams by $n$, so that they fit in the unit
square $\oi^2$ with a corner at $(0,1)$, then the lower boundary is the
graph of the empirical distribution function of the corresponding
normalized values, \ie, the distribution function of $\nu(G)$ or
$\nul(G)$.
Hence, these  distribution functions are related by reflection in  the
diagonal between $(0,1)$ and $(1,0)$, so by \eqref{musss} (and the
comment after it), for any \tg{} $G$, 
\begin{equation*}
  \nul(G)=\nu(G)\sss.
\end{equation*}

\begin{acks}
Large parts of this research was done during visits of SJ to 
Universit\'e de Nice  and  of PD and SH to Uppsala
University in January and March 2007, partly funded by the ANR Chaire
d'excellence to  PD. Work was continued during a visit of SJ to the Institut
Mittag-Leffler, Djursholm, Sweden, 2009.
SH was supported by grants NSF DMS-02-41246 and NIGMS R01GM086884-2.
We thank Adam Guetz and Sukhada Fadnavis for careful reading of a preliminary draft.
\end{acks}

\bibliographystyle{abbrv}
\bibliography{threshold}

\end{document}